\documentclass[10pt,reqno]{amsart}
\usepackage{amsmath,amsthm,amssymb,mathrsfs,stmaryrd,color}
\usepackage[all]{xy}
\usepackage{url}
\usepackage{geometry}
\usepackage[utf8]{inputenc}
\usepackage[T1]{fontenc}

\usepackage{relsize}
\usepackage[bbgreekl]{mathbbol}
\usepackage{amsfonts}

\DeclareSymbolFontAlphabet{\mathbb}{AMSb}
\DeclareSymbolFontAlphabet{\mathbbl}{bbold}
\newcommand{\Prism}{\mathbbl{\Delta}}
\usepackage{enumitem}
\setlist[enumerate]{itemsep=2pt,parsep=2pt,before={\parskip=2pt}}

\usepackage[colorlinks=true,hyperindex, linkcolor=magenta, pagebackref=false, citecolor=cyan,pdfpagelabels]{hyperref}

\newcommand{\cosimp}[3]{\xymatrix@1{#1 \ar@<.4ex>[r] \ar@<-.4ex>[r] & {\ }#2 \ar@<0.8ex>[r] \ar[r] \ar@<-.8ex>[r] & {\ } #3 \ar@<1.2ex>[r] \ar@<.4ex>[r] \ar@<-.4ex>[r] \ar@<-1.2ex>[r] & \cdots }}

\newcommand{\colim}{\mathop{\mathrm{colim}}}

\newcommand{\adjunction}[4]{\xymatrix@1{#1{\ } \ar@<0.3ex>[r]^{ {\scriptstyle #2}} & {\ } #3 \ar@<0.3ex>[l]^{ {\scriptstyle #4}}}}

\begin{document}

\newtheorem{theorem}{Theorem}[section]
\newtheorem*{theorem*}{Theorem}
\newtheorem*{definition*}{Definition}
\newtheorem{proposition}[theorem]{Proposition}
\newtheorem{lemma}[theorem]{Lemma}
\newtheorem{corollary}[theorem]{Corollary}

\theoremstyle{definition}
\newtheorem{definition}[theorem]{Definition}
\newtheorem{question}[theorem]{Question}
\newtheorem{remark}[theorem]{Remark}
\newtheorem{warning}[theorem]{Warning}
\newtheorem{example}[theorem]{Example}
\newtheorem{notation}[theorem]{Notation}
\newtheorem{convention}[theorem]{Convention}
\newtheorem{construction}[theorem]{Construction}
\newtheorem{claim}[theorem]{Claim}
\newtheorem{assumption}[theorem]{Assumption}

\newcommand{\qc}{q-\mathrm{crys}}

\newcommand{\Shv}{\mathrm{Shv}}
\newcommand{\et}{{\acute{e}t}}
\newcommand{\crys}{\mathrm{crys}}
\renewcommand{\inf}{\mathrm{inf}}
\newcommand{\Hom}{\mathrm{Hom}}
\newcommand{\Sch}{\mathrm{Sch}}
\newcommand{\Spf}{\mathrm{Spf}}
\newcommand{\Spa}{\mathrm{Spa}}
\newcommand{\Spec}{\mathrm{Spec}}
\newcommand{\perf}{\mathrm{perf}}
\newcommand{\qsyn}{\mathrm{qsyn}}
\newcommand{\perfd}{\mathrm{perfd}}
\newcommand{\arc}{{\rm arc}}
\newcommand{\conj}{\mathrm{conj}}
\newcommand{\rad}{\mathrm{rad}}

\newcommand{\psh}{\mathrm{PShv}}
\newcommand{\scr}{\mathrm{sCAlg}}
\newcommand{\HT}{\mathrm{HT}}
\newcommand{\dR}{\mathrm{dR}}
\newcommand{\Syn}{\mathrm{Syn}}
\newcommand{\gr}{\mathrm{gr}}
\newcommand{\cN}{\mathcal{N}}
\newcommand{\Fil}{\mathrm{Fil}}

\setcounter{tocdepth}{1}

\title{Crystals and chern classes}
\author{Bhargav Bhatt}

\begin{abstract}
The goal of this paper is to study the Chern classes of coherent sheaves (and more generally perfect complexes) that admit crystal structures in the setting of crystalline cohomology and more generally relative prismatic cohomology. In the former theory, we show that the Chern classes vanish on the nose; in the latter theory, we show the classes are torsion with uniformly bounded exponents determined by suitable Bernoulli numbers. We also formulate some questions about syntomic Chern classes of such sheaves.
\end{abstract}

\maketitle

\tableofcontents

\section{Introduction}
\label{intro}

Given a smooth variety $X/\mathbf{C}$ and a coherent sheaf $E \in \mathrm{Coh}(X)$, the (Betti) Chern classes of $E$ are a sequence of invariants $c_i(E) \in H^{2i}(X, \mathbf{Z})$ for $i > 0$; these invariants play an important role in understanding the geometry of $E$ and $X$. If $E$ admits a flat connection (which forces $E$ to be a vector bundle), these classes are torsion or equivalently $0$ as elements of $H^{2i}(X, \mathbf{Z}) \otimes \mathbf{C} \simeq H^{2i}_{dR}(X/\mathbf{C})$; an algebraic proof is recalled in \S \ref{sec:char0}.  The goal of this paper is to study the analogous question in positive and mixed characteristic. 

For the rest of this paper, fix a prime number $p$.

\subsection{Positive characteristic}

Let $k$ be a perfect field of characteristic $p$ and let $X/k$ be a smooth variety. For a coherent sheaf on $X$, the notion of a flat connection is not nearly as restrictive as it is in characteristic $0$, e.g., the Frobenius pullback of {\em any} coherent sheaf (e.g., the $p$-th power of a line bundle) admits a natural flat connection. Instead, a good characteristic $p$ counterpart of the characteristic $0$ notion of a flat connection is provided by Grothendieck's notion of a crystal \cite{GrothendieckCrystalsDix}. We shall prove that coherent sheaves on $X$ admitting crystal structures have vanishing crystalline Chern classes, answering a question of Esnault:

\begin{theorem}[Crystalline Chern classes of crystals vanish, Theorem~\ref{IsoCrysChernVanish}]
\label{MainThmCrysIntro}
Let $k$ be a perfect field of characteristic $p$,  let $X/k$ be a smooth variety, and let $M \in \mathrm{Coh}(X)$ be a coherent sheaf on $X$. If $M$ lifts to a $p$-torsionfree crystal of coherent sheaves on the crystalline site $(X/W(k))_\crys$, then for all $i > 0$, we have
\[ c_i(M) = 0 \quad \text{in} \quad H^{2i}_\crys(X).\]
\end{theorem}

In fact, we deduce this result from a stronger statement concerning crystals of perfect complexes (Theorem~\ref{CrysChernVanish}); the passage to the more general statement is essential to our argument. In the important case where $M$ is a vector bundle, Theorem~\ref{MainThmCrysIntro} was proven by Esnault--Shiho \cite{EsnaultShihoChern} (see also \cite{EsnaultShihoConv} for convergent isocrystals). Unlike characteristic $0$, a crystal structure on a coherent sheaf does not force it to be a vector bundle, so Theorem~\ref{MainThmCrysIntro} does not reduce to the vector bundle case. In fact, such non-(locally free) crystals arise naturally in geometry, e.g., there exist proper smooth maps $f:Y \to X$ of smooth varieties such that coherent sheaf underlying the $p$-torsionfree coherent crystal $\left(R^i f_* \mathcal{O}_{(Y/W(k))_\crys}\right)/\text{torsion}$ is not locally free for some $i$\footnote{An explicit such family of surfaces (and $i=3$) was found by Shizhang Li and Sasha Petrov; it relies on degenerating an isomorphism between finite flat group schemes to the $0$ map, and mimics the behaviour seen in \cite[\S 6]{LiLiuTorsion}.}; see also \cite{EsnaultShihoEx} for an explicit example of an $M$ as in Theorem~\ref{MainThmCrysIntro} which is not locally free.

Despite the above discussion, our proof of Theorem~\ref{MainThmCrysIntro} involves vector bundles critically: we use the Quillen--McDuff--Segal group completion theorem to describe the cohomology of the presheaf $K(\mathrm{Perf}(-))$ of $K$-theory of perfect complexes in terms of the cohomology of the stack $\mathrm{BGL}_\infty$ of vector bundles in a locally ringed topos (Corollary~\ref{HomologyK}); this is applied to both the Zariski and crystalline topoi to obtain the theorem.

It is curious that Theorem~\ref{MainThmCrysIntro} is an integral statement, i.e., we do not need to invert $p$ to see the vanishing. In contrast, the corresponding classical result over $\mathbf{C}$ only holds rationally (Example~\ref{Torsc1Ex}). The two phenomena can at least be related by working in mixed characteristic, as we explain in \S \ref{ss:MixedCharIntro}. Before turning to that case, however, let us highlight a question concerning a potential refinement of Theorem~\ref{MainThmCrysIntro}:

\begin{question}[Do syntomic Chern classes of crystals vanish?]
Fix $k$, $X$, and $M$ as in Theorem~\ref{MainThmCrysIntro} with $k=\overline{k}$ and $X$ proper. Are the syntomic Chern classes $c_i^\Syn(M) \in H^{2i}_\Syn(X, \mathbf{Z}_p(i))$ also $0$ for $i > 0$? This question is motivated by analogy with Reznikov's theorem \cite{ReznikovBloch} in complex geometry and holds true for $M$'s coming from geometry; we refer to Remark~\ref{SynVanCharpRemark} for a more thorough discussion. For now, we simply note that the $i=1$ case (where an affirmative answer is known) admits a classical formulation:
\begin{corollary}[Petrov--Vologodsky,  Corollary~\ref{LineBunCrys}]
Fix an algebraically closed field $k$ of characteristic $p$, a smooth proper variety $X/k$, and a line bundle $L \in \mathrm{Pic}(X)$. If $L$ lifts to a $p$-torsionfree crystal on the crystalline site $(X/W(k))_\crys$, then $L$ admits a compatible system of $p$-power roots in $\mathrm{Pic}(X)$. 
\end{corollary}
\end{question}

\subsection{Mixed characteristic}
\label{ss:MixedCharIntro}

Prismatic cohomology \cite{BMS1,BMS2,BhattScholzePrisms} is a generalization of crystalline cohomology for algebraic varieties in mixed characteristic. The essential novel feature distinguishing the prismatic theory from the crystalline theory is the former's close  relationship to the \'etale topology of the generic fibre. Nonetheless, the formal setup is quite similar to the crystalline theory. In particular, Grothendieck's notion of a crystal has a counterpart in the prismatic context given by $\varphi$-twisted prismatic crystals\footnote{The $\varphi$-twist appears naturally in relating prismatic and crystalline cohomology; it corresponds to the fact that prismatic cohomology over crystalline prisms provides a Frobenius descent of crystalline cohomology. We work with this notion in the introduction to emphasize the parallel with the crystalline result in Theorem~\ref{MainThmCrysIntro}. But the more fundamental concept is that of a prismatic crystal, and indeed we prove the result in that generality (Theorem~\ref{ChernPrismCrysTors}).} (see \S \ref{TwistedPrismaticCrystals}). The prismatic counterpart of Theorem~\ref{MainThmCrysIntro} is then the following statement:

\begin{theorem}[Prismatic Chern classes of crystals are torsion, Corollary~\ref{TwistedPrismaticCrystalsCor}]
\label{MainThmPrismIntro}
Let $(A,I)$ be a bounded prism, let $X$ be a smooth $p$-adic formal $A/I$-scheme, and let $M \in \mathrm{Perf}(X)$ be a perfect complex on $X$. If $M$ lifts to a $\varphi$-twisted prismatic crystal on prismatic site $(X/A)_\Prism$, then for all $i > 0$ and for all\footnote{The fact that we prove a vanishing result modulo all powers of $p$ but not integrally is a technical problem caused by the non-separatedness of prismatic cohomology groups and our methods; this can be ignored most cases of interest (Remark~\ref{BoundedTors}). In fact, we expect the vanishing to be true integrally in general.} $n \geq 0$, we have
\[ w(\mathbf{Z}_p,i) c_i(E) = 0 \quad  \text{in} \quad H^{2i}_\Prism(X/A; \mathbf{Z}/p^n)\{i\};\]
here $w(\mathbf{Z}_p,i)$ is the $p$-part of the denominator of $B_i/i$, where $B_i$ is the $i$-th Bernoulli number (this integer divides $2ip$; see Definition~\ref{defwint} for the exact formula). 
\end{theorem}

In fact, we deduce this result from a stronger statement where $M$ can be any Hodge--Tate crystal of perfect complexes on $X$ (Theorem~\ref{ChernPrismCrysTors}). The factors $w(\mathbf{Z}_p,i)$ appearing in Theorem~\ref{MainThmPrismIntro}, which arise for us naturally via Galois cohomology, are absent in Theorem~\ref{MainThmCrysIntro}, making the former result {\em prime facie} weaker than the latter. However, these factors are not just artifacts of the proof but in fact are forced by nature: the prismatic Chern class specializes to the \'etale Chern class upon taking generic fibres, and there are many examples of vector bundles with flat connection where the latter is torsion but nonzero, so the former can at best be torsion.  The following example captures this phenomenon (see Example~\ref{Chernexnonzero} for more details):

\begin{example}[A torsion Chern class for a prismatic crystal]
\label{ex:TorsionChernIntro}
Let $p=2$. Fix a complete and algebraically closed extension $C/\mathbf{Q}_2$, and let $(A,I)$ be the Fontaine prism $(A_{\inf}(\mathcal{O}_C), \ker(\theta))$. Let $Y/\mathcal{O}_C$ be a K3 surface admitting a fixed point free involution, and let $f:Y \to X:=Y/(\mathbf{Z}/2)$ be the corresponding $\mathbf{Z}/2$-torsor, so $X$ is an Enriques surface; such examples exist by work of Lang and Ogus, see \cite[Theorems 1.3 \& 1.4]{LangEnriques1} and also \cite[\S 2.1]{BMS1}. Then $L = f_* \mathcal{O}_Y/\mathcal{O}_X  \in \mathrm{Pic}(X)$ lifts to a $\varphi$-twisted prismatic crystal on $(X/A)_\Prism$ as $f$ is finite \'etale, whence  $c_1(L) \in H^2_\Prism(X/A)\{1\}$ is a $2$-torsion element by Theorem~\ref{MainThmPrismIntro} (since $w(\mathbf{Z}_2,1) = 2$). However, this class is non-trivial: in fact, its image in $H^2_{et}(X_C, \mathbf{Z}_p(1)) \otimes_{\mathbf{Z}_p} A$ under the \'etale specialization to the generic fibre is simply $c_1^{et}(L_C) \otimes 1$, which nonzero by a standard  calculation (Example~\ref{Torsc1Ex}).
\end{example}

\begin{remark}
Example~\ref{ex:TorsionChernIntro} also illustrates a consequence of Theorem~\ref{MainThmPrismIntro}: if a vector bundle $V$ with flat connection on $X_C$ (notation as Example~\ref{ex:TorsionChernIntro}) lifts to a $\varphi$-twisted prismatic crystal on $X$, then the torsion order of $c_i(V) \in H^{2i}(X_C, \mathbf{Z}_p(i))$ is bounded by $w(\mathbf{Z}_p,i)$; this gives a non-trivial obstruction for the existence of such prismatic lifts (Example~\ref{NoPrismaticLift}). The fact that this obstruction vanishes for bundles of geometric origin was already known to Grothendieck; see Remark~\ref{RelatedTorsion}.
\end{remark}

We end by pointing to Remark~\ref{SynMixedQuestion}, which asks for a version of Theorem~\ref{MainThmPrismIntro} in syntomic cohomology when $M$ admits a lift to a prismatic $F$-gauge on $X$ (or, ambitiously, merely to an absolute prismatic crystal). 

\subsection*{Acknowledgements}
We are most grateful to Jacob Lurie for many discussions: he was a collaborator at all stages of this project, but declined to sign the paper as a coauthor. We are also grateful to H\'el\`ene Esnault for a number of discussions and encouragement. In fact, Theorem~\ref{MainThmCrysIntro} was proven directly in response to a question she posed to the author at the 2015 Algebraic Geometry mega-conference in Utah; the generalization to the prismatic case in Theorem~\ref{MainThmPrismIntro} partially inspired the discovery of  the prismatic logarithm, which then sidetracked us into a different project \cite{BhattLurieAPC,BhattLuriePrismatization},  explaining (somewhat) the long delay in the appearance of this paper. We are also grateful to Esnault for many comments, references and suggestions on an earlier version of the note. We would also like to thank Johan de Jong, Shizhang Li, Sasha Petrov, Gleb Terentiuk, Vadim Vologodsky and Bogdan Zavyalov for many helpful discussions.

The author was partially supported by the NSF (\#1801689, \#1952399, and \#1840234), the Packard Foundation, and the Simons Foundation.

\begin{notation}
Write $\mathrm{Vect}(-)$ and $\mathrm{Perf}(-)$ for the stack of vector bundles and perfect complexes respectively on the category of commutative rings; we regard these stacks as valued in groupoids and $\infty$-groupoids respectively. For a commutative ring $R$, let $\mathrm{Vect}(-/R)$ and $\mathrm{Perf}(-/R)$ be their restriction to $R$-algebras. Note that $\mathrm{Vect}(-/R) = \sqcup_n \mathrm{BGL}_{n}(-/R)$ as sheaves. We also write $\mathrm{Vect}_{\infty}(-/R) := \colim_n \mathrm{Vect}_n(-/R)$, where the transition maps are given by adding a rank $1$ free summand.

The de Rham cohomology of a smooth $k$-stack $X$ is defined via smooth descent from the case of schemes or algebraic spaces; thus, if $U^\bullet \to X$ is the Cech nerve of a smooth surjection $U^0 \to X$ with $U^0$ an algebraic space, we have $R\Gamma_{dR}(X) \simeq \lim R\Gamma_{dR}(U^\bullet)$ via pullback; similarly for Hodge cohomology $R\Gamma(X,\Omega^*)$. (These are also the definitions employed in \cite{TotaroHodge}; see also \cite[\S 2]{ABMHKR} for more details.)

 \end{notation}

\section{de Rham cohomology of $\mathrm{BGL}_n$}

In this section, we recall a fundamental calculation of the de Rham and Hodge cohomology of the classifying stack of vector bundles; this will be the ultimate source of our vanishing theorem in the crystalline case.

\begin{theorem}[de Rham cohomology of the stack of vector bundles]
\label{DeligneCalc}
Let $k$ be a ring. Fix an integer $n > 0$, and let $\mathrm{Vect}_n(-/k)$ be the classifying stack of vector bundles on $k$-schemes.
\begin{enumerate}
\item We have $H^i(\mathrm{Vect}_n(-/k), \Omega^j) = 0$ for $i \neq j$. In particular, the Hodge-to-de Rham spectral sequence for the de Rham cohomology (relative to $k$) of $\mathrm{Vect}_n(-/k)$ degenerates.

\item There is a natural isomorphism
\[ \eta:k[c_1,...,c_n] \to H^*_{dR}(\mathrm{Vect}_n(-/k))\]
of graded rings, where $\deg(c_i) = 2i$. Moreover, $\eta(c_i) \in \mathrm{Fil}^i_H H^{2i}_{dR}(\mathrm{Vect}_n(-/k))$.

\item The isomorphism in (2) is compatible with restriction along the map $f:\mathrm{Vect}_n(-/k) \to \mathrm{Vect}_n(-/k)$ given by adding a free summand, i.e., $f^*(c_i) = c_i$ for $i \leq n$ and $f^*(c_{n+1}) = 0$. If we set $\mathrm{Vect}_\infty(-/k) = \colim_n \mathrm{Vect}_n(-/k)$ to be the colimit along these map, then  
\[ H^*_{dR}(\mathrm{Vect}_\infty(-/k)) \simeq k[c_1,c_2,....]\]
is a graded polynomial ring in infinitely many generators with $\mathrm{deg}(c_i) = 2i$.
\end{enumerate}
\end{theorem}
\begin{proof}
All statements follow from Deligne's calculation of $H^*_{dR}(\mathrm{Vect}_n(-/k))$ via Grassmannians, as explained in \cite[\S II]{Grospadicregulator}. Alternately, (1) is immediate from \cite[Theorem 4.1]{TotaroHodge} and (2) is \cite[Lecture 14, Theorem 18]{ClausendRcourse}. For (3), one examines the calculation in (2) to conclude that $f^*$ has the asserted properties; the statement about $\mathrm{Vect}_\infty(-/k)$ follows formally as $R\Gamma_{dR}(\mathrm{Vect}_\infty(-/k)) \simeq \lim_n R\Gamma_{dR}(\mathrm{Vect}_n(-/k))$.
\end{proof}

\section{Chern class vanishing for flat connections in characteristic $0$}
\label{sec:char0}

Our goal in this section, we explain a stack-theoretic proof of the following classical result.

\begin{theorem}
\label{dRChernVanish}
Let $X/\mathbf{C}$ be a smooth variety, and let $E \in \mathrm{Vect}(X)$ be a vector bundle admitting a flat connection. Then $c_i(E) = 0 \in H^{2i}_{dR}(X)$ for $i > 0$
\end{theorem}

The standard proof of this result relies on Chern--Weil theory: this theory reconstructs the de Rham Chern classes of $E$ from the curvature of a connection on $E$, which trivially gives their vanishing when the curvature vanishes. In this section, we instead explain a purely algebraic argument via Simpson's de Rham stack formalism; the only non-formal input is the vanishing of $H^{>0}(\mathrm{BGL}_n, \mathcal{O})$ (Theorem~\ref{DeligneCalc}). 
The material plays no direct role in the rest of the paper, but it does provide a simple toy model for our later work in positive/mixed characteristic. 

First, let us recall the de Rham stack construction and its main features: 

\begin{construction}[Simpson's de Rham stack]
Fix a presheaf $Y$ (of, say, groupoids) on finitely generated $\mathbf{C}$-algebras. The de Rham stack $Y^{dR}$ is the presheaf on the same category determined by the formula $Y^{dR}(R) = Y(R_{red})$. There is a tautological map $\eta_Y:Y \to Y^{dR}$ of presheaves. Note that if $Y$ is a sheaf for the Zariski/\'etale topology, the same holds true for $Y^{dR}$.
\end{construction}

The de Rham stack is rather computable:

\begin{example}[de Rham stack calculations]
\label{dRStackC}
We explain two calculations of de Rham stacks to explain their geometric nature:
\begin{enumerate}
\item Varieties: Say $Y=\mathbf{G}_a$. Then we have $Y^{dR}(R) = R_{red} = R/\mathrm{Nil}(R)$. Noticing that $\widehat{\mathbf{G}_a}(R) = \mathrm{Nil}(R)$, we can identify $\mathbf{G}_a^{dR}$, via the map $\eta_{\mathbf{G}_a}$, as the quotient stack $\mathbf{G}_a/\widehat{\mathbf{G}_a}$, where the quotient is interpreted in presheaves (or, since $H^{>0}_{et}(-,\widehat{\mathbf{G}_a})$ vanishes on affines, equivalently in \'etale sheaves). This example admits two generalizations:
\begin{itemize}
\item If $Y$ is a smooth variety equipped with an \'etale map $Y \to \mathbf{A}^n$, then $Y^{dR} \simeq Y/\widehat{\mathbf{G}_a}^n$, where the action of the $i$-th copy of $\widehat{\mathbf{G}_a}$ corresponds to exponentiating the vector field $\frac{d}{dx_i}$, with $x_i$ being the $i$-th co-ordinate on $\mathbf{A}^n$.

\item If $G/\mathbf{C}$ is a smooth group scheme, then $G^{dR} = G/\widehat{G}$, where $\widehat{G} \subset G$ is the formal completion of $G$ at the origin. 
\end{itemize}

\item Stacks: The construction $Y \mapsto Y^{dR}$ preserves finite limits as well as \'etale surjections. Consequently, if $Y$ is an algebraic stack presented as a quotient groupoid $U/R$ in schemes, then $Y^{dR} = U^{dR}/R^{dR}$ as an \'etale sheaf. In particular, if $G/\mathbf{C}$ is a smooth group scheme with classifying stack $BG$ in the \'etale topology, then $(BG)^{dR} = B(G^{dR}) = B(G/\widehat{G}) = BG/B\widehat{G}$ as \'etale sheaves.
\end{enumerate}
\end{example}

Perhaps the most relevant feature of this construction is that it turns de Rham cohomology into coherent cohomology, and similarly for coefficients:

\begin{theorem}[Simpson]
\label{dRStackComp}
Say $Y/\mathbf{C}$ is a smooth variety.
\begin{enumerate}
\item The category $\mathrm{Vect}(Y^{dR})$ is identified with the category $\mathrm{Vect}^\nabla(Y)$ of vector bundles with flat connection on $Y$, with pullback along $\eta_Y$ amounting to forgetting the connection. 
\item If we write $(E,\nabla)^{dR} \in \mathrm{Vect}(Y^{dR})$ for the vector bundle corresponding to a pair $(E,\nabla) \in \mathrm{Vect}^\nabla(Y)$, then there is a natural identification
\[ R\Gamma_{dR}(Y, (E,\nabla)) \simeq R\Gamma(Y^{dR}, (E,\nabla)^{dR})\]
of complexes. In particular, we have $R\Gamma_{dR}(Y) = R\Gamma(Y^{dR}, \mathcal{O})$; the pullback $\eta^*:R\Gamma(Y^{dR},\mathcal{O}) \to R\Gamma(Y,\mathcal{O})$ corresponds to passage to $\mathrm{gr}^0$ of that Hodge filtration.
\end{enumerate}
\end{theorem}

The theorem also extends to smooth\footnote{In fact, the smoothness assumption can be dropped by \cite{BhattCompletions}.} stacks (for suitable definitions of flat connections in that context).  We refer to \cite{SimpsonTeleman,GaitsgoryRozenblyumCrystals} as well as \cite[\S 2.3]{Bhatt2022LecturesFGauges} for further discussion of this fascinating construction. 

\begin{proof}[Proof sketch]
As $Y$ is a smooth variety, the map $\eta_Y:Y \to Y^{dR}$ is a surjection of presheaves (and thus after sheafification for any topology) by the infinitesimal lifting property of smooth morphisms. Consequently, if $Y^{\bullet/Y^{dR}}$ denotes the Cech nerve of $\eta_Y:Y \to Y^{dR}$, we have a descent identification
\[ \mathcal{D}_{qc}(Y^{dR}) \simeq \lim \mathcal{D}_{qc}(Y^{\bullet/Y^{dR}}).\]
On the other hand, by inspection from the functor points, we can identify $Y^{\bullet/Y^{dR}}$ with the formal completion of the Cech nerve of $Y \to \mathrm{Spec}(\mathbf{C})$ along the small diagonal. Consequently, by the basic formalism of the infinitesimal site, we learn that 
\[ \lim \mathcal{D}_{qc}(Y^{\bullet/Y^{dR}}) \simeq \mathcal{D}_{qc}(Y_{\mathrm{inf}})\]
and thus that
\[ \mathcal{D}_{qc}(Y^{dR}) \simeq \mathcal{D}_{qc}(Y_{\mathrm{inf}}).\]
The theorem now follows from standard results on the infinitesimal site. 
\end{proof}

An essential feature of the de Rham stack construction is that $Y^{dR}$ is quite close to being an algebraic stack, so we can reasonably treat it as an object of algebraic geometry (see Example~\ref{dRStackC}). Importantly,  both $Y$ and $Y^{dR}$ inhabit the same world and can thus talk to each other rather easily.   Let us explain how to use this flexibility to prove Theorem~\ref{dRChernVanish}.

\begin{proof}[Proof of Theorem~\ref{dRChernVanish}]
Let $G=\mathrm{GL}_{n,/\mathbf{C}}$ where $n=\mathrm{rank}(E)$, so $BG=\mathrm{Vect}_n(-/\mathbf{C})$ is the stack of rank $n$ vector bundles over $\mathbf{C}$. The vector bundle $E$ is classified by a map $[E]:X \to BG$, which then gives rise to the map $[E]^{dR}:X^{dR} \to {BG}^{dR}$ on de Rham stacks. By the comparison in Theorem~\ref{dRStackComp} as well as the calculation in Theorem~\ref{DeligneCalc}, we have universal Chern classes $c_1,....,c_n \in H^{2i}(\mathrm{BG}^{dR}, \mathcal{O}) \simeq H^{2i}_{dR}({BG})$ such that 
\begin{equation}
\label{eq:ccdR}
 c_i(E) = ([E]^{dR})^* c_i \in H^{2i}(X^{dR}, \mathcal{O}) \simeq H^{2i}_{dR}(X).
 \end{equation}
Using the dictionary from Theorem~\ref{dRStackComp}, the given flat connection $\nabla$ on $E$ amounts to descending $E$ along $\eta_X$, i.e. to specifying a map $\alpha_E:X^{dR} \to {BG}$ together with an identification of the composite
\[ X \xrightarrow{\eta_X} X^{dR} \xrightarrow{\alpha_E} {BG} \]
with the classifying map $[E]$. Given this structure,  the functor of points description identifes the composition
\[ X^{dR} \xrightarrow{\alpha_E} {BG} \xrightarrow{\eta_{{BG}}} {BG}^{dR} \]
 with $[E]^{dR}$. In particular, the pullback
\[ ([E]^{dR})^*:H^*({BG}^{dR},\mathcal{O}) \to H^*(X^{dR},\mathcal{O})\]
factors over $H^*({BG},\mathcal{O})$, which vanishes in positive degrees by Theorem~\ref{DeligneCalc}; this implies the vanishing of $c_i(E)$ for $i > 0$ by the formula \eqref{eq:ccdR}.
\end{proof}

In the sequel, we will give more elaborate versions of the above argument, first in crystalline cohomology and then in prismatic cohomology. Before ending this section, for completeness, let us note a standard example showing one cannot obtain a stronger integral statement over $\mathbf{C}$:

\begin{example}[A torsion Chern class for a flat bundle]
\label{Torsc1Ex}
Fix an integer $n > 0$. Let $X$ be a complex analytic space with fundamental group $G=\mathbf{Z}/n$ (such as an Enriques surface if $n=2$). Picking a primitive $n$-th root of unity gives an injective homomorphism $\chi:G \to \mathbf{C}^*$, which in turn corresponds to a rank $1$ local system $L$ on $X$; this gives rise to a line bundle with flat connection $\mathcal{L} = L \otimes_{\mathbf{C}} \mathcal{O}_X$.  We claim that the first Chern class $c:= c_1(\mathcal{L}) \in H^2(X,\mathbf{Z})$ is nonzero and in fact its image $\overline{c} \in H^2(X,\mathbf{Z}/n)$ is already nonzero.  Indeed, by comparing the Kummer sequences for the $n$-power map on the constant sheaf $\mathbf{C}^*$ and the sheaf $\mathcal{O}_X^*$ of invertible functions on $X$, we would learn that if $\overline{c} =0$, then the rank $1$ local system $L$ admits an $n$-th root;  this is not possible as the character $\chi$ of $G$ does not admit an $n$-th root.
\end{example}

\section{Crystalline Chern classes for perfect complexes}
\label{sec:ChernCons}

Our goal in this section is to construct Chern classes for perfect complexes in crystalline cohomology (Construction~\ref{conscryschernperf}). We shall first construct them for vector bundles using Theorem~\ref{DeligneCalc} in \S \ref{ss:ChernVect}, and then extend to perfect complexes using $K$-theory in \S \ref{ss:ChernPerf}. But first, with an eye towards later use, we record in \S \ref{ss:CrysFun} a procedure to compute crystalline cohomology in terms of a ``crystallization'' functor; this is essentially a reformulation of the classical approach to crystalline cohomology via the big crystalline site where we emphasizing working with (pre)sheaves in a fixed topos rather than letting the topos vary. Our basic setup for crystalline cohomology is the following.

\begin{notation}[The base category of PD-algebras]
\label{not:crys}
Fix a $p$-adically complete and $p$-torsionfree ring $W$, and let $k = W/p$. The category of $p$-complete $W$-algebras is endowed with the $p$-complete Zariski topology. Let $\mathrm{PDAlg}_{/W}$ be the category of pairs $(D,J)$, where $D$ is a $p$-nilpotent $W$-algebra and $J \subset D$ is an ideal containing $p$ and equipped with a divided power structure compatible with the one on $pD$. We equip this category with the topology induced by the Zariski topology on $D$ (or equivalently $D/J$ as $D \to D/J$ has a locally nilpotent kernel). The presheaves $\mathcal{O}_\crys, \mathcal{J}_\crys, \overline{\mathcal{O}_\crys}$ are given by the functors $(D,J) \mapsto D, J, D/J$ respectively; they are all Zariski sheaves as the notion of a divided power structure is Zariski local. 
\end{notation}

\begin{remark}[Relation to big crystalline sites]
\label{RelatePDAlgBigCrys}
The category $\mathrm{PDAlg}_{/W}$ is closely related to the big crystalline site as discussed in the Stacks project \cite[Tag 0715]{StacksProject}: the opposite of $\mathrm{PDAlg}_{/W}$ coincides with the full subcategory of affine objects in the big crystalline site $\mathrm{CRIS}(\mathrm{Spec}(k)/(W, (p), \mathrm{can}))$ of $\mathrm{Spec}(k)$ relative to the PD-ring obtained from $W$ by by giving $pW$ its unique  PD-structure. 

Similarly, given a $k$-scheme $X$ (or in fact any presheaf on $k$-algebras), the full subcategory of affine objects in the big crystalline site $\mathrm{CRIS}(X/(W,(p),\mathrm{can}))$ identifies with the category $\mathcal{C}_X$ of pairs $(D,J) \in \mathrm{PDAlg}_{/W}$ equipped with a map $\mathrm{Spec}(D/J) \to X$. For future use, let us write 
\[ R\Gamma_\crys(X/W) := R\Gamma(\mathrm{CRIS}(X/(W,(p),\mathrm{can}), \mathcal{O}_\crys)\]
 and call it the crystalline cohomology\footnote{For smooth $k$-schemes $X$,  it suffices to work with the small crystalline site $\mathrm{Cris}(X/(W,(p),\mathrm{can}))$ for the purpose of constructing/studying crystalline cohomology and crystals. However, we will need to contemplate the crystalline cohomology when $X$ is more general (e.g., an Artin stack); in such cases, the small crystalline site is typically too small, and one must instead work with the big crystalline site to obtain the correct invariants.} of $X/W$. 
\end{remark}

\subsection{The crystallization functor}
\label{ss:CrysFun}

Given a $k$-scheme or $k$-stack $X$, one classically constructs its category of crystals and crystalline cohomology using the big crystalline site $\mathrm{CRIS}(X/(W,(p),\mathrm{can}))$ (see second paragraph of Remark~\ref{RelatePDAlgBigCrys}).  Let us explain a reformulation of these notions which avoids contemplating the big crystalline site directly except in the case of a point.

\begin{construction}[The crystallization functor]
Let $X$ be a presheaf on $k$-algebras. Write $X^\crys$ for presheaf on $\mathrm{PDAlg}_{/W}$ defined by 
\[ X^\crys(D,J) = X(D/J) = \mathrm{Map}_k(\mathrm{Spec}(D/J),X).\] 
When $X$ is a sheaf for the Zariski topology, so is $X^\crys$. Unless otherwise specified, we shall write $(X^\crys, \mathcal{O}_\crys)$ for the slice ringed topos $(\mathrm{Shv}(\mathrm{PDAlg}_{/W})_{/X^{\crys}},\mathcal{O}_\crys)$. We shall write $\mathrm{Perf}(X^\crys, \mathcal{O}_\crys)$ for the $\infty$-category of crystals of perfect complexes, i.e.,
\[ \mathrm{Perf}(X^\crys, \mathcal{O}_\crys) := \lim_{\mathcal{C}_X} \mathrm{Perf}(D),\]
where $\mathcal{C}_X$ is the category of points of $X^\crys$, i.e., the category of pairs $(D,J) \in \mathrm{PDAlg}_{/W}$ equipped with a $k$-map $\mathrm{Spec}(D/J) \to X$. 
\end{construction}

\begin{example}
When $X=\mathbf{G}_a$ on $k$-algebras, we have $X^\crys = \overline{\mathcal{O}_\crys}$. 
\end{example}

The following result is essentially a definition, but we state it separately for ease of reference and also because similar constructions appear later  in the prismatic context (Construction~\ref{Cons:Prismatize}):

\begin{proposition}
\label{cryscohstack}
For any presheaf $X$ on $k$-algebras, there is a natural identifcation between $\mathrm{Perf}(X^\crys,\mathcal{O}_\crys)$ and crystals of perfect complexes on $X$. In particular, the complex $R\Gamma(X^\crys, \mathcal{O}_\crys)$ identifies naturally with the crystalline cohomology $R\Gamma_\crys(X/W)$ of $X$ relative to $W$.
\end{proposition}
\begin{proof}
For the first statement, recall that the $\infty$-category of crystals of perfect complexes on $X$ is, by definition, given by the inverse limit
\[ \lim_{(U,T,\delta) \in \mathrm{CRIS}(X/(W, (p), \text{can})} \mathrm{Perf}(T).\]
By Zariski descent for perfect complexes, it suffices to compute the limit over affine objects of $\mathrm{CRIS}(X/(W, (p), \text{can})$, which is precisely our definition of  $\mathrm{Perf}(X^\crys,\mathcal{O}_\crys)$.  The second statement is deduced by computing endomorphisms of the unit crystal on either side.
\end{proof}

\subsection{Chern classes for vector bundles}
\label{ss:ChernVect}

In this subsection, we recall the construction of Chern classes for vector bundles in crystalline cohomology using Theorem~\ref{DeligneCalc}; our definitions are formulated in terms of the crystallization functor from \S \ref{ss:CrysFun}.

\begin{construction}[Chern classes for vector bundles]
\label{ChernVect}
Given a presheaf $X$ on $k$-algebras, a vector bundle $E$ on $X$, and an integer $i > 0$, our goal is to construct a natural map
\[ X^\crys \xrightarrow{c_i(E)} \mathcal{O}_\crys[2i]\]
of presheaves of spaces on $\mathrm{PDAlg}_{/W}$; its homotopy class is an element of $H^{2i}_\crys(X/W)$, and is called the $i$-th Chern class $c_i(E)$ of $E$. 

First, let us construct it in the universal case. For each $n \geq 0$, we claim that there are natural identifications of graded rings
\[ H^*_{\crys}(\mathrm{Vect}_n(-/k)/W) \simeq H^*_{dR}(\mathrm{Vect}_n(-/W)/W)^{\wedge} \simeq W[c_1,c_2,....,c_n],\]
where $\deg(c_i) = 2i$; the first equality uses the standard fact that crystalline cohomology can be computed as the $p$-completed de Rham cohomology of a flat $W$-lift, and the second equality is Theorem~\ref{DeligneCalc}. Taking the inverse limit over $n$ gives an isomorphism of graded rings
\[ H^*_{\crys}(\mathrm{Vect}_\infty(-/k)/W) \simeq W[c_1,c_2,...,c_n,...],\]
where $\deg(c_i) = 2i$. Using the identification $R\Gamma_\crys(-/W) \simeq R\Gamma( (-)^\crys, \mathcal{O}_\crys)$ from Lemma~\ref{cryscohstack}, we obtain, for each $i > 0$, well-defined homotopy classes of maps
\[ \mathrm{Vect}_n(-/k)^\crys \to \mathcal{O}_\crys[2i] \quad \text{and} \quad \mathrm{Vect}_\infty(-/k)^\crys \to \mathcal{O}_\crys[2i]\]
with the second one factoring the first (where we work in sheaves of spaces on $\mathrm{PDAlg}_{/W}$); we call either of these the (universal) $i$-th Chern class map.

We can now specialize the above universal construction. If $X$ is a $k$-scheme and $E$ is a rank $n$ vector bundle on $X$, then $E$ is classified by a map $[E]:X \to \mathrm{Vect}_n(-/k)$.  Crystallizing this map and composing with the universal $c_i$ yields a map
\[ X^\crys \xrightarrow{[E]^\crys} \mathrm{Vect}_n(-/k)^\crys \xrightarrow{c_i} \mathcal{O}_\crys[2i],\]
and thus a cohomology class $c_i(E) \in H^{2i}_\crys(X)$; we call this the $i$-th Chern class of $E$.
\end{construction}

\subsection{Chern classes for perfect complexes}
\label{ss:ChernPerf}

We now extend the Chern classes from \S \ref{ss:ChernVect} to perfect complexes using algebraic $K$-theory. Before diving into the construction, let us briefly explain the strategy\footnote{It is perhaps worth explaining why this subsection is not reduced to the single sentence that ``the $K$-group of vector bundles and perfect complexes coincide on a regular scheme.'' The primary reason is that we formulate statements at the space level, since we later wish to apply them to the crystallization functor.}. First, one constructs Chern classes for elements in $K$-theory using the group completion theorem in algebraic $K$-theory. Next, one applies the Thomason-Trobaugh theorem implying that the $K$-theory of perfect complexes and vector bundles on affines agree, thus giving Chern classes of perfect complexes on affines. Finally, as our constructions are actually performed at the space level (instead of being at the level of cohomology classes), one can globalize to perfect complexes on any $X$.

Let us begin by recalling the Thomason-Trobaugh approach to $K$-theory via perfect complexes and its relation to Quillen's classical construction.

\begin{construction}[$K$-theory via vector bundles or perfect complexes]
\label{KVectPerf}
For a commutative ring $R$, define $K(R)$ to be the Thomason-Trobaugh (connective) $K$-theory of the stable $\infty$-category $\mathrm{Perf}(R)$ of perfect $R$-complexes on $X$ as in \cite{ThomasonTrobaugh}.	 By  
\cite[3.10]{ThomasonTrobaugh}, this agrees with Quillen's $K$-theory space $K(R)$ attached to the exact category $\mathrm{Vect}(R)$ of vector bundles on $R$ as in \cite{QuillenK1}. As $\mathrm{Vect}(R)$ is a split exact category, its $K$-theory is given by the group completion of the $E_\infty$-space $\mathrm{Vect}(R)^{\simeq}$ given by the groupoid underlying $\mathrm{Vect}(R)$ with symmetric monoidal structure given by direct sum, i.e., the natural map
\[ \mathrm{Vect}(R)^{\simeq} \to K(R)\]
of $E_\infty$-spaces identifies $K(R)$ with the group completion $\Omega B \mathrm{Vect}(R)^{\simeq}$ (see \cite[Theorem 7.1]{WeibelKBook}). Sheafifying for the Zariski topology and using the Zariski sheaf property of $K$-theory, we learn that the natural map
\[ \sqcup_n B(\mathrm{GL}_n(-)) \to K(-)\]
of presheaves on commutative rings becomes a group completion after sheafifying the source. 
\end{construction}

To use the above effectively, we shall need to control the cohomology of the $K$-theory presheaf. This control is provided by the group completion theorem, whose statement we recall next, with a brief hint of the proof:

\begin{theorem}[Quillen, McDuff-Segal]
\label{McDuffSegal}
Fix a commutative ring $R$, and let $M = \sqcup_n B(\mathrm{GL}_n(R))$ be the displayed $E_\infty$-space. Then the group completion $M^{grp} := \Omega BM$ of $M$ comes equipped with a homology equivalence $M_\infty := \mathbf{Z} \times B(\mathrm{GL}_\infty(R)) \to M^{grp}$, functorially in $R$.
\end{theorem}
\begin{proof}
There is a natural action of the $E_\infty$-space $M$ on the space $M_\infty$: realize $M_\infty$ as the direct limit of $M \xrightarrow{T} M \xrightarrow{T} M \xrightarrow{T} ...$, where $T$ is the map obtained by ``adding a free $R$-module of rank $1$''. The  geometric realization of the resulting map $M_\infty/M \to BM$ of simplicial spaces is a homology fibration whose total space $|M_\infty/M|$ is contractible, which yields a homology equivalence $M_\infty \to \Omega BM$ (see \cite[Proposition 2]{McDuffSegal} for more precise statements, and Quillen's \cite[Appendix Q]{FriedlanderMazur} for another perspective).
\end{proof}

Using the group completion theorem, we can relate the cohomology of the $K$-theory presheaf to that of the stack of vector bundles quite directly:

\begin{corollary}[Cohomology of the $K$-theory presheaf]
\label{HomologyK}
Let $(\mathcal{C},\mathcal{O})$ be a locally ringed topos. Then there is a natural map
\[ \mathbf{Z} \times B(\mathrm{GL}_\infty(\mathcal{O}(-))) \to K(\mathcal{O}(-)) \]
of presheaves of spaces inducing a homology isomorphism locally on $\mathcal{C}$. In particular, for any abelian sheaf $A$ on $\mathcal{C}$, pullback along the above induces an isomorphism 
\[ R\Gamma(K(\mathcal{O}(-)), A) \simeq R\Gamma(\mathbf{Z} \times B(\mathrm{GL}_\infty(\mathcal{O}(-))),A).\]
\end{corollary}
\begin{proof}
The first part follows from Theorem~\ref{McDuffSegal} as $(\mathcal{C}, \mathcal{O})$ is {\em locally} ringed. For the second part, it suffices to prove the following general statement:

\begin{itemize}
\item[$(\ast)$] If $X \to Y$ is a map of presheaves of spaces on $\mathcal{C}$ that induces a homology isomorphism locally on $\mathcal{C}$, then $R\Gamma(Y,A) \simeq R\Gamma(X,A)$ for any abelian sheaf $A$. 
\end{itemize}

The assumption implies that the map $\mathbf{Z}[X] \to \mathbf{Z}[Y]$ of $\mathcal{D}(\mathrm{Ab})$-valued presheaves induces an isomorphism on homology presheaves locally on $\mathcal{C}$, so its homotopy fibre $F$ is a $\mathcal{D}(\mathrm{Ab})$-valued presheaf whose homology is locally $0$ on $\mathcal{C}$. But then $F$ has no maps into truncated $\mathcal{D}(\mathrm{Ab})$-valued sheaves on $\mathcal{C}$ (such as $A[i]$ for any $i \in \mathbf{Z}$), so the claim follows.
\end{proof}

Putting the above ingredients together, we obtain our promised construction of Chern classes for perfect complexes:

\begin{construction}[Chern classes on the stack of perfect complexes]
\label{conscryschernperf}
For each $i > 0$,  we shall construct a natural (up to homotopy) map
\[ c_i:\mathrm{Perf}(-/k)^{\crys} \to \mathcal{O}_{\crys}[2i]\] 
of presheaves on $\mathrm{PDAlg}_{/W}$, yielding Chern classes for perfect complexes.  

First, let us construct a Chern class map from $K$-theory itself.  For any commutative ring $R$, write $K(-/R)$ be $K$-theory presheaf (Construction~\ref{KVectPerf}) on $R$-algebras. By Corollary~\ref{HomologyK} (applied with $\mathcal{O} = \overline{\mathcal{O}_\crys}$ and $A=\mathcal{O}_\crys$) and the fact that $B(\mathrm{GL}_\infty(-))$ sheafifies to $\mathrm{Vect}_\infty(-)$, we have a natural isomorphism
\[ R\Gamma( K(-/k)^\crys, \mathcal{O}_{\crys}) \simeq R\Gamma( \mathbf{Z} \times \mathrm{Vect}_{\infty}(-/k)^\crys, \mathcal{O}_{\crys})\]
of $E_\infty$-rings. On cohomology rings, this yields an identification
\[ H^*(K(-/k)^\crys, \mathcal{O}_{\crys}) \simeq H^*( \mathbf{Z} \times \mathrm{Vect}_{\infty}(-/k)^\crys, \mathcal{O}_{\crys}) = \prod_{\mathbf{Z}} k[c_1,c_2,...,c_n,....]\]
as graded commutative rings.  In particular, each $c_i$, viewed diagonally inside the target, yields a homotopy class of maps
\[ K(-/k)^{\crys} \to \mathcal{O}_{\crys}[2i]\]
of sheaves of $\infty$-groupoids on $\mathrm{PDAlg}_{/W}$.

On the other hand, there is also the universal  map $\mathrm{Perf}(-) \to K(-)$ of presheaves on commutative rings. Restricting to $k$-algebras, crystallizing, and composing with the above map yields a map
\[ c_i:\mathrm{Perf}(-/k)^\crys \to K(-/k)^\crys \to \mathcal{O}_{\crys}[2i]\]
of presheaves of spaces on $\mathrm{PDAlg}_{/W}$. Consequently, just as in Construction~\ref{ChernVect}, given any scheme $X$ equipped with a perfect complex $E$, the preceding map determines Chern classes $c_i(E) \in H^{2i}_\crys(X)$.

\end{construction}

\begin{remark}[Chern classes for perfect complexes in other cohomology theories]
\label{rmk:synChern}
In Construction~\ref{conscryschernperf}, we defined Chern classes for perfect complexes in crystalline cohomology. Beyond the classical theory of Chern classes for vector bundles, the essential inputs are that perfect complexes and vector bundles give rise to the same notion of $K$-theory after Zariski sheafification, and that crystalline cohomology is a Zariski sheaf of spectra. Thus, it adapts readily to other cohomology theories admitting Chern classes for vector bundles and satisfying Zariski descent as well, e.g., the syntomic Chern classes for vector bundles constructed in \cite{BhattLurieAPC} naturally extend to all perfect complexes by this procedure.
\end{remark}

\section{Vanishing of Chern classes of crystals in crystalline cohomology}
\label{sec:ChernClassCrys}

In this section, we prove the vanishing of Chern classes in crystalline cohomology for perfect complexes that underlie crystals (Theorem~\ref{CrysChernVanish}). Note that the case of locally free sheaves was previously shown by Esnault--Shiho \cite{EsnaultShihoChern}, and our argument uses similar ideas about classifying spaces for crystalline vector bundles; the key additional input comes from $K$-theory, enabling us to pass to perfect complexes. We shall use the notation and constructions used in \S \ref{sec:ChernCons}.

To begin, let us explain how to formulate the notion of a crystal of vector bundles or perfect complexes using the formalism of functors on $\mathrm{PDAlg}_{/W}$. The essential input is the following construction:

\begin{construction}[The $\sharp$-functor]
\label{SharpDef}
Any presheaf $X$ on $W$-algebras induces a presheaf $X^\sharp$ on $\mathrm{PDAlg}_{/W}$ via $X^\sharp(D,J) = X(D)$. If $X$ is a sheaf for the $p$-complete Zariski topology, then $X^\sharp$ is a sheaf for the Zariski topology. There is a natural map $\eta_X:X^\sharp \to (X_k)^\crys$ given on the functor of points by the natural map 
\[ X^\sharp(D,J ) := X(D) \to (X_k)^\crys(D,J) := X(D/J)\] 
for any $(D,J) \in \mathrm{PDAlg}_{/W}$.
\end{construction}

\begin{example}
When $X=\mathbf{G}_a$ over $W$, we have $X^\sharp = \mathcal{O}_\crys$.
\end{example}

We can now explain how to formulate the notion of a crystal using the above construction $Y \mapsto Y^\sharp$ applied to $Y=\mathrm{Vect}(-/W)$ being the stack of vector bundles over $W$.

\begin{remark}[Crystal structures via the $\sharp$ functor]
\label{CrystalStructureFunctor}
Fix a presheaf $X$ on $k$-algebras. 

\begin{enumerate}
\item {\em The groupoid of crystals of vector bundles:} The groupoid $\mathrm{Map}(X^\crys, \mathrm{Vect}(-/W)^\sharp)$ of maps $X^\crys \to \mathrm{Vect}(-/W)^\sharp$ identifies with the groupoid of crystals of vector bundles on $X$: specifying a map $X^\crys \to \mathrm{Vect}(-/W)^\sharp$ is equivalent to functorially specifying a bundle $E_D \in \mathrm{Vect}(D)$ for each $(D,J) \in \mathcal{C}_X$,  which is exactly the definition of a crystal of vector bundles on the big crystalline site of $X/W$.

\item {\em Crystal structures on a given vector bundle:} Given a vector bundle $E$ with classifying map  $[E]:X \to \mathrm{Vect}_n(-/k)$, the groupoid of factorizations of $[E]^\crys:X^\crys \to \mathrm{Vect}_n(-/k)^\crys$ through $\eta_{\mathrm{Vect}(-/W)}:\mathrm{Vect}(-/W)^\sharp \to \mathrm{Vect}(-/k)^\crys$ is identified with the groupoid of crystal structures on $E$. Indeed,  the data of such a factorization amounts to the following: for each pair $(D,J) \in \mathrm{PDAlg}_{/W}$ and a map $\alpha:\mathrm{Spec}(D/J) \to X$ (i.e., an object of $\mathcal{C}_X$), one must functorially specify a lift of $\alpha^*(E) \in \mathrm{Vect}(D/J)$ along $D \to D/J$; this is exactly the data of a lift of $E$ to a crystal on the big crystalline site of $X/W$. 
\end{enumerate}

A similar comment applies to $\mathrm{Vect}(-/W)$ replaced by other functors on $W$-algebras (e.g., the stack $\mathrm{Perf}(-/W)$ of perfect complexes). 
\end{remark}

To effectively use the above reformulations in applications, we shall need to understand the  cohomology of the $\sharp$-construction; this is recorded next for the cases of interest:

\begin{proposition}
\label{cohcohstack}
Let $X$ be a presheaf on $W$-algebras which admits a hypercover in the \'etale topology by $W$-schemes with (locally) bounded $p^\infty$-torsion.  Then the complex $R\Gamma(X^\sharp,\mathcal{O}_\crys)$ identifies naturally with $R\Gamma(X,\mathcal{O}_X)^{\wedge}$.
\end{proposition}

A typical $X$ allowed in the proposition is a $W$-flat algebraic stack: the Cech nerve of a smooth atlas is a hypercover with the required properties.

\begin{proof}
As divided power structures lift uniquely along \'etale covers by \cite[Tag 07H1]{StacksProject}, the construction carrying $X$ to $X^\sharp$ preserves \'etale covers. Moreover, this construction commutes with finite limits by definition. Thus, it suffices to prove the claim for $X = \mathrm{Spec}(R)$ with $R$ being a $W$-algebra with bounded $p$-torsion, i.e., we must show that $R\Gamma(\mathrm{Spec}(R)^\sharp, \mathcal{O}_\crys) \simeq R^{\wedge}$, where the target is the derived $p$-completion of $R$ (which coincides with the classical $p$-completion by the boundedness assumption). By definition, we have $R\Gamma(\mathrm{Spec}(R)^\sharp, \mathcal{O}_\crys) = \lim_{\mathcal{C}'_R} D$, where $\mathcal{C}'_R$ is the category of PD-pairs $(D,J) \in \mathrm{PDAlg}_{/W}$ equipped with a map $R \to D$ of $W$-algebras.  But $\mathcal{C}'_R$ has a final collection of objects given by $(R/p^n R, pR/p^nR)$ thanks to \cite[Tag 07H1]{StacksProject}, so $\lim_{\mathcal{C}'_R} D \simeq \lim_n R/p^n R \simeq R^{\wedge}$, as wanted.
\end{proof}

Using the preceding reformulation of the notion of a crystal, the following is the key lemma that ultimately gives our desired vanishing theorem for Chern classes:

\begin{lemma}
\label{KeyVanishingChernCrys}
For each $i > 0$, the composition 
\[ \mathrm{Perf}(-/W)^\sharp \xrightarrow{\eta_{\mathrm{Perf}(-/W)}} \mathrm{Perf}(-/k)^\crys \xrightarrow{c_i} \mathcal{O}_{\crys}[2i]\] 
is  null-homotopic.
\end{lemma}
\begin{proof}
The composition under consideration factors as
\[ \mathrm{Perf}(-/W)^\sharp \xrightarrow{\eta_{\mathrm{Perf}(-/W)}} \mathrm{Perf}(-/k)^\crys \xrightarrow{\mathrm{can}^\crys}  K(-/k)^\crys \xrightarrow{c_i} \mathcal{O}_{\crys}[2i]\]
by construction of Chern classes for perfect complexes. By functoriality of the map $\eta$ in Construction~\ref{SharpDef}, it also factors as 
\[ \mathrm{Perf}(-/W)^\sharp \xrightarrow{\mathrm{can}^\sharp} K(-/W)^\sharp \xrightarrow{\eta_{K(-/k)}} K(-/k)^\crys \xrightarrow{c_i} \mathcal{O}_{\crys}[2i].\]
It therefore suffices to show that all maps $K(-/W)^\sharp \to \mathcal{O}_{\crys}[j]$ are null-homotopic for $j > 0$. Using Corollary~\ref{HomologyK},  this amounts to showing that $H^j R\Gamma(\mathbf{Z} \times \mathrm{Vect}_{\infty}(-/W)^\sharp, \mathcal{O}_\crys) = 0$ for $j >0$. By Proposition~\ref{cohcohstack} and the Kunneth formula, it is enough to show that $H^j R\Gamma(\mathrm{Vect}_{\infty}(-/W), \mathcal{O})^{\wedge} = 0$. As $\mathrm{Vect}_\infty(-/W) = \colim_n \mathrm{Vect}_n(-/W)$, it is enough to note that $W \simeq H^*(\mathrm{Vect}_n(-/W), \mathcal{O})$ by Theorem~\ref{DeligneCalc}.
\end{proof}

Our main theorem in crystalline cohomology is the following:

\begin{theorem}[Chern classes of crystals vanish]
\label{CrysChernVanish}
Let $X$ be a $k$-scheme. If a perfect complex $E$ on $X$ lifts to a crystal of perfect complexes on $(X/W)_\crys$, then $c_i(E) = 0$ for $i > 0$.
\end{theorem}
\begin{proof}
A perfect complex on $X$ gives a map $[E]:X \to \mathrm{Perf}(-/k)$ and hence a map $[E]^\crys:X^\crys \to \mathrm{Perf}(-/k)^\crys$; composing with $c_i:\mathrm{Perf}(-/k)^\crys \to \mathcal{O}_{\crys}[2i]$ gives us $c_i(E)$. Specifying a lift of $E$ to a crystal of perfect complexes on $X$ relative to $W$ is equivalent to factoring the map $[E]^\crys$ as a map $X^\crys \to \mathrm{Perf}(-/W)^\sharp \to \mathrm{Perf}(-/k)^\crys$ (see Remark~\ref{CrystalStructureFunctor}). The vanishing claim now follows from Lemma~\ref{KeyVanishingChernCrys}.
\end{proof}

Theorem~\ref{CrysChernVanish} implies, in particular, that the Chern classes of any crystal of locally free sheaves on $(X/W)_\crys$ must vanish. In practice, the local freeness is restrictive, and a more flexible class of objects is provided by torsionfree lattices in isocrystals\footnote{As pointed out in \cite{EsnaultShihoChern}, it is not known whether every isocrystal on a smooth $k$-scheme possesses a locally free lattice; while there are some situations where this holds true for non-obvious geometric reasons (see, e.g., \cite[Theorem 3.5 and \S 4]{EsnaultShihoConv}), we suspect it is not true in general.}, or equivalently by a $p$-torsionfree crystal in coherent sheaves. We can deduce a vanishing result of such crystals from Theorem~\ref{CrysChernVanish} by a slight elaboration of the fact that any coherent sheaf on a smooth variety is a perfect complex:

\begin{corollary}
\label{IsoCrysChernVanish}
Let $X$ be a smooth $k$-scheme.  Let $E$ be a $p$-torsionfree crystal in  coherent sheaves on $X$. Then $c_i(E(X)) = 0$ for $i > 0$.
\end{corollary}

The proof below looks notationally complicated, partially as we have not assumed that $X$ comes embedded into projective space.  The arguments becomes particularly transparent in the stacky-approach, see Remark~\ref{StackyPerftoCoh}.

\begin{proof}
By Theorem~\ref{CrysChernVanish}, it is enough to show that $E$ promotes naturally to a crystal $\widetilde{E}$ of perfect complexes on $X$ such that $E(X) = \widetilde{E}(X)$. Recall that the relevant categories of crystals are defined as 
\[ \mathrm{Perf}(X^\crys,\mathcal{O}_\crys) = \lim_{\mathcal{C}_X} \mathrm{Perf}(D) \quad \text{and} \quad  \mathrm{Coh}(X^\crys,\mathcal{O}_\crys) = \lim_{\mathcal{C}_X} \mathrm{Mod}^{fp}(D).\]
We shall consider the following subcategory of the LHS:
\[  \mathrm{Perf}^{cn}(X^\crys,\mathcal{O}_\crys) = \lim_{\mathcal{C}_X} \mathrm{Perf}^{cn}(D),\]
where $\mathrm{Perf}^{cn}(D) \subset \mathrm{Perf}(D)$ refers to the full subcategory of connective perfect complexes. Taking $H^0$ of a connective perfect complex gives a finitely presented module, and this process commutes with base change, so we have a natural functor
\[ h:\mathrm{Perf}^{cn}(X^\crys,\mathcal{O}_\crys) \to \mathrm{Coh}(X^\crys, \mathcal{O}_\crys).\]
Our given $p$-torsionfree crystal $E$ is an object of the right hand side. Any lift $E'$ of $E$ along $h$ must satisfy $H^0(E'(X)) = E(X)$ by connectivity of $E'$. Thus, it suffices to show the following:
\begin{itemize}
\item[$(\ast)$] 
The functor $h$ identifies the full subcategory $\mathcal{C} \subset  \mathrm{Perf}^{cn}(X^\crys, \mathcal{O}_\crys)$ spanned those $F$'s such that $F(X)$ is concentrated in degree $0$ with the full subcategory of $\mathrm{Coh}(X^\crys, \mathcal{O}_\crys)$ spanned by $p$-torsionfree crystals.
\end{itemize}
This assertion can be checked locally and after making choices, so we may  assume $X=\mathrm{Spec}(R)$ is a smooth affine $k$-scheme equipped with a $p$-completely smooth $W$-algebra $\widetilde{R}$ lifting $R$. Let $D^\bullet$ be the $p$-completed pd-envelope of the diagonal in the Cech nerve of $W \to \widetilde{R}$; by smoothness, the face maps in $D^\bullet$ are $p$-completely flat, and hence each face map $\widetilde{R} = D^0 \to D^i$ is honestly flat as $\widetilde{R}$ is noetherian.  Moreover,  standard arguments in the crystalline theory show that evaluation of crystals yields equivalences
\[ r_{\mathrm{Perf}^{cn}}: \mathrm{Perf}^{cn}(X^\crys, \mathcal{O}_\crys) \simeq \lim \mathrm{Perf}^{cn}(D^\bullet) \quad \text{and} \quad r_{\mathrm{Coh}}: \mathrm{Coh}(X^\crys, \mathcal{O}_\crys) \simeq \lim \mathrm{Mod}^{fp}(D^\bullet).\]
Under the first equivalence, evaluating $F \in  \mathrm{Perf}^{cn}(X^\crys, \mathcal{O}_\crys)$ at $X=\mathrm{Spec}(R)$ corresponds to taking a cartesian cosimplicial perfect complex $F^\bullet = r_{\mathrm{Perf}^{cn}}(F) \in  \lim \mathrm{Perf}^{cn}(D^\bullet)$ to $F^0 \otimes_{\widetilde{R}}^L R \in \mathrm{Perf}(R)$; similarly for the second equivalence, except the derived tensor product is replaced by an ordinary tensor product. Thus, using the derived Nakayama lemma and the regularity of $\widetilde{R}$, we can identify the category $\mathcal{C}$ in $(\ast)$ as the full subcategory of $F^\bullet \in \lim_{D^\bullet} \mathcal{D}^{\leq 0}(D^\bullet)$ such that $F^0$ is a $p$-torsionfree finitely presented $\widetilde{R}$-module. But then each $F^i$ must be a $p$-torsionfree finitely presented $D^i$-module for all $i$ (as any face map $D^0 \to D^i$ is flat). Thus, we have identified $\mathcal{C} \subset  \mathrm{Perf}^{cn}(X^\crys, \mathcal{O}_\crys)$  with 
\[ \lim_{D^\bullet} \mathrm{Mod}^{fp}(D^\bullet)',\]
where $\mathrm{Mod}^{fp}(D^i)' \subset \mathrm{Mod}^{fp}(D^i)$ is the full subcategory of $p$-torsionfree objects; by the same flatness arguments again, this limit is exactly the full subcategory of  $\mathrm{Coh}(X^\crys, \mathcal{O}_\crys)$ spanned by $p$-torsionfree crystals, so we have shown $(\ast)$.
\end{proof}

\begin{remark}[Independence of the lattice]
\label{LattIndep}
Given $p$-torsionfree crystals $E,F \in \mathrm{Coh}(X^\crys, \mathcal{O}_\crys)$, if $E$ and $F$ are identified in the isogeny category $\mathrm{Coh}(X^\crys, \mathcal{O}_\crys)[1/p]$ (i.e., they give the same isocrystal), then the coherent sheaves $E(X)$ and $F(X)$ on $X$ have the same class in $K(X)$ (as explained in the next paragraph); consequently, Corollary~\ref{IsoCrysChernVanish} can be seen as a vanishing result for Chern classes of isocrystals. (The argument below also appears in \cite[Proposition 3.1]{EsnaultShihoConv}.)

To see the statement in $K$-theory, we may assume (by multiplying by a $p$-power) that $E$ is a suboject of $F$ with cokernel $Q=F/E$ killed by $p^n$. Filtering $Q$ by the $p$-adic filtration, we may assume $n=1$, i.e., $Q$ is a crystal in coherent sheaves on $X/W$ which is actually killed by $p$, i.e., is a crystal of coherent sheaves on $X/k$ or equivalently a coherent sheaf on $X$ with a flat connection with nilpotent $p$-curvature. Evaluating the exact triangle $E \to F \to Q$ in $\mathrm{Perf}(X^\crys, \mathcal{O}_\crys)$ on $X$ itself shows that $E(X) \to F(X)$ has both kernel and cokernel given by the coherent sheaf $Q(X)$, whence $[E(X)] = [F(X)]$ in $K$-theory, as wanted. 
\end{remark}

In the rest of this section, we make a few extended remarks. First, we reformulate Theorem~\ref{CrysChernVanish} and Corollary~\ref{IsoCrysChernVanish} using the stack theoretic perspective on crystalline cohomology.

\begin{remark}[A stacky reformulation of vanishing]
\label{stackycrysvan}
In \S \ref{sec:char0}, we explained the characteristic $0$ version of Theorem~\ref{CrysChernVanish} using Simpson's de Rham stack construction. The proof of Theorem~\ref{CrysChernVanish} given above avoided discussing the analog of Simpson's construction, and thus looked slightly different. Nevertheless, using the recent prismatization functors studied in \cite{DrinfeldPrismatization, BhattLurieAPC,BhattLuriePrismatization,Bhatt2022LecturesFGauges} (which give an analog of Simpson's construction for the crystalline/prismatic theory), one can give a proof of Theorem~\ref{CrysChernVanish} that looks quite similar to the one  \S \ref{sec:char0}. We briefly explain this approach as it avoids contemplating functors on the category $\mathrm{PDAlg}_{/W}$ of PD-pairs, naturally suggests an analog over other prisms (explored in the sequel), and also hints at what is special about crystalline prisms (namely, the ``crystalline'' map $\eta$ in the proof below). To avoid complications, we formulate the statements only for vector bundles, though they can be extended to all perfect complexes using $K$-theory just as before. For the rest of this remark, we assume familiarity with the formalism of prismatization.

Let $X$ be a smooth $k$-scheme equipped with a rank $n$ vector bundle $E$. Assume that $E$ admits a crystal structure. By \cite{BhattLuriePrismatization}, such a structure is equivalent to specifying a vector bundle on $X^\Prism$ lifting $E$ along the de Rham map $\rho_{dR}:X \to X^\Prism$. More symbolically, such a structure is encoded by a map $\alpha_E:X^\Prism \to \mathrm{Vect}_n(-/W)^{\wedge}$ of $p$-adic formal $W$-stacks together with an identification of the composition 
\[ X \xrightarrow{\rho_{dR}} X^\Prism \xrightarrow{\alpha_E} \mathrm{Vect}_n(-/W)^{\wedge} \]
with the composition
\[ X \xrightarrow{[E]} \mathrm{Vect}_n(-/k) \to \mathrm{Vect}_n(-/W)^{\wedge},\]
where the first map is the classifying map for the bundle $E$ and the second map is the tautological map. Since $\mathrm{Vect}_n(-/W)$ is a $W$-stack lifting $\mathrm{Vect}_n(-/k)$, we have a ``crystalline'' map $\eta:\mathrm{Vect}_n(-/W)^{\wedge} \to \mathrm{Vect}_n(-/k)^\Prism$: this is induced via transmutation from the natural map $\epsilon_R:R \to W(R)/p$ for any $p$-nilpotent ring $R$, induced by the Frobenius on $W(R)$. Now, using the functoriality of $\alpha_E$ under $\epsilon_R$ for a $p$-nilpotent test ring $R$, one verifies\footnote{This relies on the following somewhat funny observation about animated rings: for a $p$-nilpotent animated ring $R$, the two maps $W(R)/p \to W(W(R)/p)/p$ given by $\epsilon_{W(R)/p}$ and $W(\epsilon_R)/p$ are canonically homotopic, where, as before, $\epsilon_S:S \to W(S)/p$ is the natural map induced by $F:W(S) \to W(S)$ for any animated ring $S$.} the following: the composition 
\[ X^\Prism \xrightarrow{\alpha_E} \mathrm{Vect}_n(-/W)^{\wedge} \xrightarrow{\eta} \mathrm{Vect}_n(-/k)^\Prism\]
identifies with 
\[ [E]^\Prism:X^\Prism \to \mathrm{Vect}_n(-/k)^\Prism.\]
In particular, the pullback on $\mathcal{O}$-cohomology for $[E]^\Prism$ factors as
\[ \left([E]^\Prism\right)^*: H^*(\mathrm{Vect}_n(-/k)^\Prism, \mathcal{O}) \to H^*(\mathrm{Vect}_n(-/W)^{\wedge},\mathcal{O}) \to H^*(X^\Prism,\mathcal{O}),\]
and is thus $0$ in positive degrees as the term in the middle is $0$ in positive degrees by Theorem~\ref{DeligneCalc}. On the other hand, by functoriality, the Chern classes $c_i(E)$ are pulled back from the universal Chern classes $c_i(E) \in H^{2i}(\mathrm{Vect}_n(-/k)^\Prism, \mathcal{O})$ along $[E]^\Prism$; it follows that $c_i(E) = 0$ for $i > 0$.
\end{remark}

\begin{remark}[The stacky perspective on the passage from perfect complexes to coherent sheaves]
\label{StackyPerftoCoh}
Corollary~\ref{IsoCrysChernVanish} is also quite clear from the stacky perspective, so let us sketch the argument. As $X$ is smooth, its prismatization $X^\Prism$ is a smooth $p$-adic formal $W$-stack such that $\mathrm{Perf}(X^\Prism) \simeq \mathrm{Perf}(X^\crys, \mathcal{O}_\crys)$ and $\mathrm{Coh}(X^\Prism) \simeq \mathrm{Coh}(X^\crys, \mathcal{O}_\crys)$ by \cite{BhattLuriePrismatization}; the regularity of $X^\Prism$ gives an inclusion $\mathrm{Coh}(X^\Prism) \subset \mathrm{Perf}(X^\Prism)$, letting us view a crystal of coherent sheaves as a crystal of perfect complexes. Moreover, we have a natural map $\eta:X \to X^\Prism$ such that pullback along $\eta$ corresponds to evaluating a crystal on $X$ (with derived pullback used for perfect complexes, and underived pullback used for coherent sheaves). Critically, the map $\eta$ factors as $X \to (X^\Prism)_{p=0} \to X^\Prism$, where the first map is faithfully flat and the second map is the tautological one. It is then clear if $E \in \mathrm{Coh}(X^\Prism)$ is $p$-torsionfree, then $\eta^* E = L\eta^* E$, so we can view such an $E$ as a crystal of perfect complexes without changing its value on $X$. Corollary~\ref{IsoCrysChernVanish} then follows from Theorem~\ref{CrysChernVanish}.
\end{remark}

We end this section by discussing a question on vanishing of syntomic Chern classes.

\begin{remark}[A question on vanishing of syntomic Chern classes]
\label{SynVanCharpRemark}
Fix a smooth proper $k$-scheme $X$. One then has the attached syntomic cohomology\footnote{Syntomic cohomology is a form of \'etale motivic cohomology, defined recently in full generality. Its relationship to crystalline cohomology is analogous to the relationship of Deligne cohomology with singular cohomology over the complex numbers. In the characteristic $p$ case considered here, it also coincides with the cohomology of logarithmic de Rham--Witt sheaves. We refer to \cite[Remark 2.14]{BhattICM22} for further references.} complex $R\Gamma_\Syn(X,\mathbf{Z}_p(i))$ for each $i \geq 0$, which sits in a fibre sequence
\[ R\Gamma_\Syn(X,\mathbf{Z}_p(i)) \to \mathrm{Fil}^i_{\mathcal{N}} R\Gamma_\crys(X/W) \xrightarrow{\varphi_i - 1} R\Gamma_\crys(X/W),\]
where $\mathrm{Fil}^*_{\mathcal{N}}$ is the Nygaard filtration on crystalline cohomology, and $\varphi_i$ is a certain divided Frobenius map heuristically given by $\frac{\varphi}{p^i}$. As explained in Remark~\ref{rmk:synChern}, given a perfect complex $E$ on $X$, one has syntomic Chern classes $c_i^\Syn(E) \in H^{2i}_\Syn(X,\mathbf{Z}_p(i))$, refining those in crystalline cohomology. Motivated by Theorem~\ref{CrysChernVanish}, one can ask:

\begin{itemize}
\item[$(\ast)$] If $E \in \mathrm{Perf}(X)$ lifts to a crystal in perfect complexes in $(X/W)_\crys$, then do the syntomic Chern classes $c_i^\Syn(E) \in H^{2i}_\Syn(X, \mathbf{Z}_p(i))$ also vanish for $i > 0$? 
\end{itemize}

Let us note right away that the question has a negative answer for $i=1$ if $k$ is not algebraically closed (see item (4) below); it is probably most reasonable to ask $(\ast)$ over an algebraically closed field $k$.  We next make a sequence of remarks on this question (which is already interesting when $E$ is a vector bundle); we thank Esnault for a number of discussions surrounding $(\ast)$ as well as the remarks below.

\begin{enumerate}
\item {\em The motivating complex analogue:}  Recall that Bloch conjectured \cite{Blochdilog} that the Chow-theoretic $i$-th Chern class $c_i$ of a vector bundle with flat connection on a smooth projective variety $M/\mathbf{C}$ lies in the deepest layer $F^i \mathrm{CH}^i(M)_{\mathbf{Q}}$ of the conjectural Bloch--Beilinson filtration on the rational Chow group $\mathrm{CH}^i(M)_{\mathbf{Q}}$. Under the interpretation of the Chow group as an appropriate $\mathrm{Hom}$-group in a category of motives, this filtration is induced by the conjectural ``standard'' $t$-structure on the category of motives over $\mathbf{C}$. In particular, under any realization functor towards a $\mathbf{Q}$-linear category $\mathcal{D}$ that is already equipped with a ``standard'' $t$-structure, this conjectural filtration on $\mathrm{CH}^i(M)_{\mathbf{Q}}$ ought to map to the filtration on $\mathrm{Hom}_{\mathcal{D}}(\mathbf{Q}(0), \mathcal{H}(M)(i)[2i])$ induced by the Postnikov filtration on the motive $\mathcal{H}(M)$ of $M$. Thus, Bloch's conjecture implies that such a $c_i$ class must map to $0$ if $i$ exceeds the cohomological dimension $c$ of $\mathcal{D}$. We have $c=1$ for the realization coming from Hodge theory, so Bloch's conjecture predicts the following:

\begin{theorem}[Reznikov \cite{ReznikovBloch}]
\label{Reznikov}
Fix a smooth projective variety $M/\mathbf{C}$ and a flat vector bundle $V \in \mathrm{Vect}^\nabla(M)$.  For $i \geq 2$, the Deligne cohomology Chern classes $c_i^D(V) \in H^{2i}_D(M, \mathbf{Q}(i))$ are $0$. 
\end{theorem}

Regarding syntomic cohomology as an internal Hom in the realization of $p$-adic motives provided by $F$-gauges then suggests viewing $(\ast)$ as a syntomic counterpart of Theorem~\ref{Reznikov}. Of course, Question $(\ast)$ is formulated at the $p$-integral level, while Reznikov's theorem is only a rational statement; we think this is reasonable since Theorem~\ref{CrysChernVanish} is also an integral statement, unlike its characteristic $0$ counterpart which is only a rational statement. Note that Reznikov's result necessarily excludes the case $i=1$: the weight $1$ Deligne cohomology group $H^2_D(M,\mathbf{Z}(1))$ is the Picard group itself. On the other hand, due to the $p$-completions implicit in syntomic cohomology, question $(\ast)$ actually has a positive answer for $i=1$ over algebraically closed fields, as discussed in (4) below.

\item {\em The case of Frobenius structures: } The question $(\ast)$  has a positive answer in the following (overlapping) cases:
\begin{itemize}
\item If $E$ admits the structure of a mod $p$ $F$-gauge on $X$ (i.e., comes from a perfect complex on $(X^\Syn)_{p=0}$, following the notation of \cite{Bhatt2022LecturesFGauges}): the $F$-gauge structure can be shown to endow $E$ with two filtrations whose associated gradeds differ by Frobenius pullback, giving a canonical identification of points $[\varphi_X^* E] = [E]$ in the $K$-theory space $K(X)$, which immediately implies the vanishing of $c_i^\Syn(E) \in H^{2i}_\Syn(X,\mathbf{Z}_p(i))$ for $i > 0$ as $\varphi_X^*$ acts by $p^i$ on $H^{2i}_\Syn(X,\mathbf{Z}_p(i))$.

\item If $E \in \mathrm{Coh}(X)$ lifts to a crystal $\widetilde{E}$ in $p$-torsionfree coherent sheaves on $X$ endowed with an $F$-isocrystal structure $\varphi_{\widetilde{E}}:\varphi^* \widetilde{E}[1/p] \simeq \widetilde{E}[1/p]$: one again has $[\varphi_X^* E] = [E]$ in $K_0(X)$ as above. This can be seen by specializing to the previous item (extend $\widetilde{E}$ to an $F$-gauge via reflexive extension and reduce $p$); alternately, it can also be seen directly by noticing that, for $N \gg 0$, the cokernel of the Frobenius $\varphi_{\widetilde{E}}:\varphi^* \widetilde{E} \simeq p^N \varphi^* \widetilde{E} \to \widetilde{E}$ is killed by a power of $p$, and hence its (derived) mod $p$ reduction has trivial class in $K$-theory as in Remark~\ref{LattIndep}. 

\item If $E \in \mathrm{Coh}(X)$ lifts to a crystal $\widetilde{E}$ in $p$-torsionfree coherent sheaves on $X$ which is convergent, i.e., infinitely $F$-divisible as an isocrystal (see \cite[Remark 2.4]{EsnaultShihoConv}). (This generalizes the previous entry.) Indeed, in this case, we would learn that $[E] \in K_0(X)$ is infinitely divisible by $\varphi_X^*$, i.e., there exists $E_n \in \mathrm{Coh}(X)$ with $E_0=E$ and $[\varphi_X^* E_{n+1}] = [E_n]$ in $K_0(X)$; as $\varphi_X^*$ acts by $p^i$ on $H^{2i}_\Syn(X, \mathbf{Z}_p(i))$, one concludes by that $c_i^\Syn(E)$ admits a compatible system of $p$-divisions, and must thus be $0$ since $H^{2i}_\Syn(X, \mathbf{Z}_p(i))$ is derived $p$-complete. 
\end{itemize}

All these cases involve additional structure related to the Frobenius. Thus, they are roughly analogous to the restriction of Reznikov's theorem to flat connections underlying variations of Hodge structures; in this case, the properness assumption can be dropped from Reznikov's theorem by \cite{CorletteEsnaultClass}, which is consistent with the fact that the proofs sketched above do not use properness.

\item {\em The rational version:} When $k=\overline{k}$, question $(\ast)$ has a positive answer modulo torsion. To see this, it is enough to show that the map $H^{2i}_\Syn(X,\mathbf{Z}_p(i)) \to H^{2i}_\crys(X)$ is injective after inverting $p$. After inverting $p$, we have (by definition) a fibre sequence
\[ R\Gamma_\Syn(X, \mathbf{Z}_p(i))[1/p] \to R\Gamma_\crys(X)[1/p] \xrightarrow{\varphi_X^*-p^i} R\Gamma_\crys(X)[1/p],\]
so it is enough to show that the second map appearing above is surjective on each cohomology group, i.e., that $\varphi_X^*-p^i$ is surjective on the Dieudonne module $(V := H^j_\crys(X/W)[1/p], \varphi_V := \varphi_X^*)$. Write $\mathrm{DM}_k$ for the isogeny category of Dieudonne modules over $k$, i.e., finite dimensional vector spaces $U$ over $K=W(k)[1/p]$ equipped with an isomorphism $\varphi_K^* U \simeq U$.  By definition of $\mathrm{Hom}$'s in $\mathrm{DM}_k$, we can identify $\mathrm{coker}(\varphi_V-p^i)$ with $\mathrm{Ext}^1_{\mathrm{DM}_k}( K(-i), V)$; the latter vanishes as $\mathrm{DM}_k$ is semisimple by the algebraic closedness of $k$ and the Dieudonne--Manin classification

\item {\em The weight $1$ case (by Petrov and Vologodsky\footnote{After seeing a preliminary version of this preprint raising Question $(\ast)$ in general, Petrov and Vologodsky kindly explained to the author how to solve the weight $1$ case. All (correct) arguments recorded in this item are due to them, and we are most grateful for their help.}):} The weight $1$ case of $(\ast)$ over an algebraically closed field follows from Theorem~\ref{CrysChernVanish} and the next result:

\begin{proposition}[Illusie]
\label{IllInj}
Assume $k=\overline{k}$. Then the natural map
\[ H^2_\Syn(X,\mathbf{Z}_p(1)) \to H^2_\crys(X) \]
is injective.
\end{proposition}
\begin{proof}
The map in question factors as 
\[ H^2_\Syn(X,\mathbf{Z}_p(1)) \xrightarrow{a} H^2(X, W\Omega^{\geq 1}_X) \xrightarrow{b} H^2(X,W\Omega_X) = H^2_\crys(X), \]
where $a$ is the de Rham-Witt version of the $d\log$ map (see \cite[page 582, Corollary 3.27]{IllusiedRW} for instance), while $b$ is the natural map. By \cite[page 630, Remark 5.10]{IllusiedRW}, the  map $a$ is injective, so it is enough to show that the second map is injective. This follows from the long exact sequence associated to the triangle
\[ W\Omega^{\geq 1}_X \to W\Omega_X \to W(\mathcal{O}_X) \]
and the fact that $H^1(X,W\Omega_X) \to H^1(X, W(\mathcal{O}_X))$ is surjective \cite[page 618, Lemma 3.11]{IllusiedRW}.
\end{proof}

Now recall that syntomic cohomology in weight $1$ has a concrete description \cite[Theorem 7.5.6]{BhattLurieAPC}: indeed, we have
\[ R\Gamma_\Syn(X,\mathbf{Z}_p(1)) = R\Gamma(X_{et}, \mathbf{G}_m)^{\wedge}[-1].\]
As $X$ is smooth proper over the perfect field $k$ , the group $H^0(X,\mathbf{G}_m)$ is uniquely $p$-divisible, while $H^1(X,\mathbf{G}_m) = \mathrm{Pic}(X)$ is the extension of a finitely generated abelian group (Neron-Severi) by a divisible abelian group (an abelian variety). Consequently, the system $\{H^1(X,\mu_{p^n})\}_{n \geq 1}$ is Mittag-Leffler. This implies that the derived $p$-completion of $\mathrm{Pic}(X)$ as an abelian group coincides with the naive $p$-adic completion $\mathrm{Pic}(X)^{\wedge} = \lim_n \mathrm{Pic}(X)/p^n$,  and also that 
\[ H^2_\Syn(X,\mathbf{Z}_p(1)) = \lim_n H^2(X,\mu_{p^n}).\]
By the Kummer sequence, this yields an injective map 
\begin{equation}
\label{PictoSyn}
 \mathrm{Pic}(X)^{\wedge}  \to H^2_\Syn(X,\mathbf{Z}_p(1))
 \end{equation}
which can be (by unwinding definitions) identified with the first Chern class map. Thus, question $(\ast)$ in rank $1$ can be formulated as the following concrete assertion:

\begin{corollary}
\label{LineBunCrys}
Assume $k=\overline{k}$. If $L \in \mathrm{Pic}(X)$ lifts to a crystal on $(X/W)_\crys$, then $L$ is trivial in $\mathrm{Pic}(X)^{\wedge}$. In fact, $L$ admits a compatible system of $p$-power roots in $\mathrm{Pic}(X)$. 
\end{corollary}
\begin{proof}
The vanishing in $\mathrm{Pic}(X)^{\wedge}$ is clear from the previous discussion, Proposition~\ref{IllInj} and Theorem~\ref{CrysChernVanish}; this implies that $L$ admits $p^n$-th roots in $\mathrm{Pic}(X)$ for all $n$. To obtain a compatible system of such roots, one can then use a completely general statement: for any abelian group $A$, if $A^{\wedge}$ denotes the derived $p$-completion of $A$ as an abelian group, then the kernel of $A \to A^{\wedge}$ is the quotient of a $\mathbf{Q}_p$-vector space\footnote{This claim follows from the general fibre sequence
\[ \mathrm{RHom}(\mathbf{Q}_p,A) \to A \to \widehat{A} := \lim_n \mathrm{Kos}(A;p^n)\]
as well as the definition $A^{\wedge} := H^0(\widehat{A})$ of the derived $p$-completion as an abelian group.}.
\end{proof}

The above settles $(\ast)$ affirmatively in weight $1$ over an algebraically closed field. Over general perfect fields, $(\ast)$ has a negative answer in weight $1$:

\begin{example}
Choose a finite field $k$ and an elliptic curve $E/k$ such that $E(k)$ contains an element of order $p^2$. Identifying $E$ as $\mathrm{Pic}^0(E)$, we obtain a line bundle $L \in \mathrm{Pic}(X)$ of order $p^2$. We shall show that $M=L^{\otimes p}$ admits a crystal structure relative to $W(k)$ and is nonzero in $\mathrm{Pic}(E)^{\wedge}$.

By deformation theory, we can choose a lift $\mathcal{E}/W(k)$ of $E$ and a lift $\mathcal{L} \in \mathrm{Pic}(\mathcal{E})$ of $L$. Now $\mathcal{L}$ has degree $0$ on the special (and hence generic) fibre of $\mathcal{E}/W$, so the Hodge Chern class $c_1(\mathcal{L})$ vanishes on the generic fibre and hence also in the $p$-torsionfree group $H^1(\mathcal{E}, \Omega^1_{\mathcal{E}/W})$. Consequently, we can choose a (necessarily flat) connection $\nabla_1$ on $\mathcal{L}$ for $\mathcal{E}/W$. Taking the $p$-th tensor power gives a flat connection $\nabla_2$ on $\mathcal{M} = \mathcal{L}^{\otimes p}$; moreover, this latter connection has trivial $p$-curvature modulo $p$ (as this is true for the $p$-th tensor power of any line bundle with flat connection on $E/k$). Consequently, the pair $(\mathcal{M},\nabla_2)$ provides a crystal structure (relative to $W$) on the line bundle $M = L^{\otimes p}$ on $E$. On the other hand, by assumption on the order of $L$, the line bundle $M \in \mathrm{Pic}(E)$ has order $p$. As $k$ is finite, the group $\mathrm{Pic}(E)$ is finitely generated, so $M$ must also have nonzero image in the $p$-adic completion $\mathrm{Pic}(E)^{\wedge}$, and thus also in $H^2_\Syn(E,\mathbf{Z}_p(1))$ by the injectivity of \eqref{PictoSyn} explained above. Thus, the bundle $M$ on $E$ provides a negative answer to $(\ast)$ in weight $1$ over a finite field.
\end{example}

\end{enumerate}

\end{remark}

\section{Cohomology of jet spaces}

This section is preparatory in nature: the goal is to record a calculation (Theorem~\ref{torsionboundjet}) that will be used in \S \ref{sec:VanPrismChern}  to show that Chern classes of vector bundles admitting prismatic crystal structures have uniformly bounded torsion.

In this section, fix a bounded prism $(A,I)$ with face $\overline{A}$. It will be convenient for technical reasons\footnote{Strictly speaking, passage to derived algebraic geometry is not essential to our discussion of arithmetic jet spaces. Indeed, we only need to contemplate $JX$ when $X$ is the stack of vector bundles. In this case, $JX$ is classical and even pro-smooth over $A$ (Example~\ref{JetVect}).} to formulate our basic definitions in the context of derived algebraic geometry; we shall use the theory of animated $\delta$-rings from \cite[Appendix B]{BhattLuriePrismatization}. Given an animated commutative ring $R$,  any map $A \to R$ refines uniquely to a $\delta$-map $A \to W(R)$. Moreover, if $R$ is $(p,I)$-nilpotent, then the map $A \to W(R)$ turns $W(R)$ into a $(p,I)$-complete $A$-algebra by \cite[Lemma 3.3]{BhattLuriePrismatization}. This motivates the following definition:

\begin{definition}[The arithmetic jet space]
\label{defJetspace}
Given a presheaf $X$ on $(p,I)$-nilpotent animated $A$-algebras, its {\em arithmetic jet space} $JX$ is the presheaf on $(p,I)$-nilpotent animated $A$-algebras determined by 
\[ JX(R) = X(W(R)) = \mathrm{Map}_A(\mathrm{Spf}(W(R)), X).\]
The Frobenius on $W(-)$ endows $JX$ with the natural structure of a $\delta$-presheaf over $A$. There is a natural map 
\[ \pi:JX \to X\]
 over $\mathrm{Spf}(A)$, induced by the projection $W(R) \to R$; we call this the {\em structure map} for $JX$.
\end{definition}

In the language of \cite{Bhatt2022LecturesFGauges}, the presheaf $JX$ and the map $\pi$ are obtained via transmutation from the ring stack $W(\mathbf{G}_a)$ and the restriction map $W(\mathbf{G}_a) \to \mathbf{G}_a$ on $(p,I)$-nilpotent $A$-algebras.

\begin{example}[The jet space of an affine scheme]
If $X = \mathbf{A}^1_{\mathrm{Spf}(A)}$, then $JX = \mathrm{Spf}(R)$, where $R$ is the co-ordinate ring of the base change $W_{\mathrm{Spf}(A)}$ of the Witt vector scheme $W$ (or, equivalently, the $(p,I)$-completion of the free $\delta$-ring on $1$ generator over $A$). The structure map is simply the restriction map $W_{\mathrm{Spf}(A)} \to \mathbf{A}^1_{\mathrm{Spf}(A)}$ on the Witt vectors. 

More generally, if. $X = \mathrm{Spf}(S)$ for a $(p,I)$-complete animated $A$-algebra $S$, then $JX = \mathrm{Spf}(F(S))$, where $F$ is the left adjoint of the forgetful functor from $(p,I)$-complete animated $\delta$-$A$-algebras  to $(p,I)$-complete animated $A$-algebras (see \cite[Proposition A.20]{BhattLuriePrismatization} for this functor).
\end{example}

\begin{example}[The jet space of $\mathrm{BGL}_n$]
\label{JetVect}
Fix an integer $r > 0$. Let $X=\mathrm{Vect}_r(-/A)$ be the stack of rank $r$ vector bundles on $p$-nilpotent animated $A$-algebras. Then for any such ring $R$, since the animated $A$-algebra $W(R)$ is $(p,I)$-complete, we have
\[ JX(R) = \mathrm{Vect}_r(W(R)) = \lim_k \mathrm{Vect}_r(W_k(R)),\]
where the inverse limit is along the restriction maps. It follows from this description and deformation theory that the projection $\mathrm{Vect}_r(W(R)) \to \mathrm{Vect}_r(R)$ is an isomorphism on $\pi_0$; moreover, if $p=0$ on $\pi_0(R)$, the restriction maps $W_{k+1}(R) \to W_k(R)$ are square-zero extensions (see \cite[Proof of Lemma 3.2 (2)]{BhattLuriePrismatization}), so the fibre of  $\mathrm{Vect}_r(W(R)) \to \mathrm{Vect}_r(R)$ over $E \in \mathrm{Vect}_r(R)$ is given by an (infinite) iterated extension of classifying stacks of vector bundles. It follows from this description that the natural map
\[ \pi:JX \to X\]
is pro-smooth. In particular, $JX$ is $A$-flat. 
\end{example}

The relevance of this construction to the prismatic theory is via the following:

\begin{construction}[The jet-to-prismatization map] 
For any presheaf $X$ on $(p,I)$-nilpotent animated $A$-algebras, write $\overline{X} = X \times_{\mathrm{Spf}(A)} \mathrm{Spf}(\overline{A})$ for the induced presheaf on $p$-nilpotent animated $A$-algebras. Its  derived prismatization $(\overline{X}/A)^\Prism$ is the presheaf on $(p,I)$-nilpotent animated $A$-algebras given by 
\[ (\overline{X}/A)^\Prism(R) = \mathrm{Map}_{\overline{A}}(\mathrm{Spf}(\overline{W(R)}), \overline{X}),\] 
as in \cite[\S 7]{BhattLuriePrismatization}. Then there is a natural map
\[ \psi:JX \to (\overline{X}/A)^\Prism \]
of presheaves on $(p,I)$-nilpotent animated $A$-algebras,  induced via functoriality from the obvious map $W(R) \to \overline{W(R)} = W(R)/I$ of animated $A$-algebras. This induces a  inducing a pullback map
\begin{equation}
\label{jettoprism}
 \psi^*:R\Gamma_\Prism(\overline{X}/A)\{i\} := R\Gamma( (X/A)^\Prism, \mathcal{O})\{i\} \to R\Gamma(JX, \mathcal{O})\{i\}
 \end{equation}
for each integer $i$
\end{construction}

The precise goal of this section is to specialize to the case where $X=\mathrm{Vect}_r(-/A)$ is the stack of rank $r$ vector bundles, and then understand the effect of $\psi^*$ on the universal Chern classes.   To formulate our theorem, we introduce the following numbers:

\begin{definition}
\label{defwint}
For each integer $i$, define a positive integer $w(\mathbf{Z}_p,i)$ as follows:
\begin{itemize}
\item Say $p=2$. If $i$ is even, set $w(\mathbf{Z}_p,i) = 2^{v_2(i)+2}$; if $i$ is odd, set $w(\mathbf{Z}_p,i) = 2$.
\item Say $p$ is odd. If $(p-1) \nmid i$, set $w(\mathbf{Z}_p, i) = 1$; if not, then  set $w(\mathbf{Z}_p,i) = p^{v_p(i)+1}$.
\end{itemize}
Here $v_p(-)$ denotes the $p$-adic valuation.
\end{definition}

The integer $w(\mathbf{Z}_p,i)$  introduced above can be interpreted as the $p$-part of the denominator of $B_i/i$, where $B_i$ is the $i$-th Bernoulli number, see \cite[Theorem B.4]{MilnorStasheff}. On the other hand, they arise for us via an explicit interpretation in terms of Galois cohomology, see Lemma~\ref{InertiaInvts}.  Our goal in this section is to prove the following result:

\begin{theorem}
\label{torsionboundjet}
Fix an integer $r \geq 0$ and let $X = \mathrm{Vect}_r(-/A)$ be the stack of rank $r$ vector bundles.  For each $0 \leq i \leq r$, the image of the universal Chern class $c_i \in H^{2i}_\Prism(\overline{X}/A)\{i\}$ in $H^{2i}(JX, \mathcal{O}/p^n)\{i\}$ under the map \eqref{jettoprism} is killed by the positive integer $w(\mathbf{Z}_p, i)$ from Definition~\ref{defwint}  for all $n \geq 0$.
\end{theorem}

Theorem~\ref{torsionboundjet} refers to Chern classes in prismatic cohomology; these were defined in \cite[\S 9]{BhattLurieAPC}, see also the forthcoming Construction~\ref{PrismChernClassDef}. We shall prove Theorem~\ref{torsionboundjet} via an approximation argument that reduces to the case of Fontaine's $A_{\inf}$ prism, introduced next.

\begin{notation}[The Fontaine prism]
\label{FontainePrism}
Let $C = \mathbf{C}_p$ be a completed algebraic closure of $\mathbf{Q}_p$, so $C$ has a natural action of $G_{\mathbf{Q}_p} := \mathrm{Gal}(\overline{\mathbf{Q}_p}/\mathbf{Q}_p)$. Associated to $C$, we have $\mathcal{O}_C$, $\mathcal{O}_C^\flat$ and $\theta:A_{\inf} := W(\mathcal{O}_C^\flat) \to \mathcal{O}_C$ as usual. The pair $(A_{\inf}, \ker(\theta))$ is a perfect prism corresponding to the perfectoid ring $\mathcal{O}_C$. We use $d \in \ker(\theta)$ to denote a generator (though we do not use a specific choice). For $M \in D(A_{\inf})$, we write $\widehat{M}$ for the derived $(p,d)$-adic completion of $M$. 
\end{notation}

To prove Theorem~\ref{torsionboundjet} over $A_{\inf}$,  we do not use the geometry of $JX$. On the other hand, we crucially use that this $A_{\inf}$-stack is actually defined over $\mathbf{Z}_p$, which yields a natural $G_{\mathbf{Q}_p}$-action on the map. Granting this, the rest of the proof is Galois theoretic, comparing Frobenius modules over $\mathbf{F}_p$ and $\mathcal{O}_C^\flat$ with $\mathbf{Z}/p^n$-coefficients; the relevant calculations are the subject of \S \ref{ss:GaloisCoh}. The proof of Theorem~\ref{torsionboundjet} appears in \S \ref{ss:ProofTorsionBoundJet}.

\subsection{Some Galois cohomology}
\label{ss:GaloisCoh}

\begin{notation}
For a perfect ring $R$ of characteristic $p$, we reserve the notation $W_n(R)[F]$ to refer to the ring of Frobenius operators. When $R = \mathbf{F}_p$, this coincides with the polynomial ring in one variable $F$ over $\mathbf{Z}/p^n$. In general, $W_n(R)[F]$ is not commutative.
\end{notation}

We need a classification of injective Frobenius modules over $\mathrm{Spec}(\mathbf{F}_p)$ for our argument.

\begin{lemma}
\label{ClassifyInjectives}
Let $M \in \mathrm{Mod}_{\mathbf{Z}/p^n[F]}$, regarded as a Frobenius module over $\mathrm{Spec}(\mathbf{F}_p)$ with $\mathbf{Z}/p^n$-coefficients. Assume $M$ is an injective object. Then $M$ is a direct sum of objects of the following type:
\begin{enumerate}
\item An algebraic Frobenius module over $\mathrm{Spec}(\mathbf{F}_p)$, i..e, a $\mathbf{Z}/p^n[F]$-module where every element is annihilated by a monic polynomial in $F$.
\item A filtered colimit of copies of finite free $\mathbf{Z}/p^n[F]$-modules (of rank $1$, even).
\end{enumerate}
\end{lemma}
\begin{proof}
Let $R := \mathbf{Z}/p^n[F]$ and $\overline{R} =  R_{red} = \mathbf{Z}/p[F]$. Since the Frobenius on $\mathrm{Spec}(\mathbf{F}_p)$ is the identity map, these are commutative noetherian rings. Thus, each injective $R$-module is a direct sum of indecomposable injective $R$-modules. Moreover, again because $R$ is noetherian, each indecomposable injective $R$-module is (non-canonically) the injective hull $E(x)$ of the residue field $\kappa(x)$ of $R$ at a prime ideal $x \in \mathrm{Spec}(R)$ by \cite[Tag 08YA]{StacksProject}. As we have a homeomorphism $\mathrm{Spec}(R) \simeq \mathrm{Spec}(\overline{R})$, such modules fall into two types, depending on whether $x$ is closed or generic, and corresponding to (1) and (2) in the lemma:
\begin{enumerate}
\item For a closed point $s \in \mathrm{Spec}(R)$ corresponding to a monic polynomial $g = g(F) \in R$ that is irreducible modulo $p$, the injective hull $E(s)$ identifies with $R[1/g]/R$. In particular, each element of $E(s)$ is killed by a positive power of $g$ in this case, so $E(s)$ is an algebraic Frobenius module. 
\item For the generic point $\eta \in \mathrm{Spec}(R)$,  the injective hull $E(\eta)$ is the total ring of fractions of $R$, i.e., $E(\eta)$ is obtained by inverting all monic polynomials in $R$. In particular, $E(\eta)$ is a filtered colimit of copies of $R$ in this case.
\end{enumerate}
To prove these statements, we need a construction of the injective hull. For this, recall that the injective hull of the residue field at a prime ideal $\mathfrak{p}$ of a Gorenstein ring $S$ can be computed as (an approrpiate shift) of  $R\Gamma_{\mathfrak{p}S_{\mathfrak{p}}}(S_{\mathfrak{p}})$. The description of local cohomology as the fibre of a map between localizations then gives the desired statements in (1) and (2).
\end{proof}

Consider the the functor $G:D(\mathbf{Z}/p^n[F]) \to D(W_n(\mathcal{O}_C^\flat)[F])$ given by 
\[ M \mapsto G(M) := M \widehat{\otimes^L_{\mathbf{Z}/p^n}} W_n(\mathcal{O}_C^\flat).\]

\begin{lemma}
\label{basechangefrob1}
The functor $G$ is $t$-exact for the usual $t$-structure. Moreover, if $M$ is a discrete $\mathbf{Z}/p^n$-module, then $G(M)$ is $d$-torsionfree.
\end{lemma}
\begin{proof}
For the first part, it is enough to show that if $M$ is discrete, then the same holds true for $G(M)$. When $n=1$, this is clear as the  derived $d$-completion of a free $\mathcal{O}_C^\flat$-module is concentrated in degree $0$. One then reduces to the case $n=1$ by filtering $M$ for its $p$-adic filtration and using the flatness of $\mathbf{Z}/p^n \to W_n(\mathcal{O}_C^\flat)$. 

For the second part, it suffices to show that $G(M) \otimes^L_{W_n(\mathcal{O}_C^\flat)} W_n(\mathcal{O}_C^\flat)/d$ is discrete for $M$ discrete. But once we kill $d$, there is no completion, so we simply have 
\[ G(M) \otimes^L_{W_n(\mathcal{O}_C^\flat)}  W_n(\mathcal{O}_C^\flat)/d \simeq M \otimes_{\mathbf{Z}/p^n}^L W_n(\mathcal{O}_C^\flat)/d\]
and the claim follows from that flatness of $\mathbf{Z}/p^n \to W_n(\mathcal{O}_C^\flat)/d$. 
\end{proof}

Using this exactness property of $G$, we can define a new functor $S:\mathrm{Mod}_{\mathbf{Z}/p^n[F])} \to \mathrm{Mod}_{\mathbf{Z}/p^n}(\mathrm{Gal}(\overline{\mathbf{Q}_p}/\mathbf{Q}_p))$ by the formula
\[ S(M) = \left(G(M)[1/d]\right)^{F=1},\]
where $(-)^{F=1}$ is interpreted in the non-derived sense.

\begin{lemma}
\label{basechangefrob2}
For any $M \in \mathrm{Mod}_{\mathbf{Z}/p^n[F]}$, the $\mathrm{Gal}(\overline{\mathbf{Q}_p}/\mathbf{Q}_p)$-module $S(M)$ is unramified.
\end{lemma}
\begin{proof}
We begin by noting two properties. The functor $S$ is left-exact (by Lemma~\ref{basechangefrob1} and left-exactness of $(-)^{F=1}$) and commutes with filtered colimits: this is clear for $G(-)$, and is proven for (the derived variant of) $\big((-)[1/d]\big)^{F=1}$ in \cite[Lemma 9.2]{BhattScholzePrisms}. Using these two properties, it suffices to prove that $S(M)$ is unramified for $M$ an injective $\mathbf{Z}/p^n[F]$-module. Using Lemma~\ref{ClassifyInjectives}, we can write $M = M_0 \oplus M_1$, with $M_0$ an algebraic Frobenius module and $M_1$ being a filtered colimit of finite free $\mathbf{Z}/p^n[F]$-modules. It suffices to prove that $S(M_0) = 0$ and $S(M_1) = 0$ separately. 

For $M_0$, as $S$ commutes with filtered colimits, it is enough to show $S(N)$ is unramified when $N$ is a holonomic Frobenius module, i.e., $N$ is finitely generated and each element of $N$ is annihilated by a monic polynomial in $F$. Such an $N$ has finite length, so the completion can be dropped in the tensor product defining $G(N)$ and hence also in the formula defining $S(N)$. We then conclude using the Riemann-Hilbert correspondence over $\mathrm{Spec}(\mathbf{F}_p)$ for \'etale sheaves of $\mathbf{Z}/p^n$-modules (see, e.g., \cite[Theorem 9.6.1]{BhattLurieRH}).

For $M_1$, by compatibility of $S$ with filtered colimits, it suffices to show that $S(\mathbf{Z}/p^n[F]) = 0$. For this, note that
\[ G(M)[1/d] = \widehat{ \big(W_n(\mathcal{O}_C^\flat)[F]\big)}[1/d] = \{ \sum_{i \geq 0} a_i F^i \in W_n(C^\flat) \llbracket F \rrbracket \mid a_i \to 0\}. \]
It remains to observe that the above module has no $F$-invariant elements: such an element would be a power series $\sum_i a_i F^i$ with $a_0 = 0$ and $a_i^p = a_{i+1}$ for all $i \geq 0$, which forces $a_i = 0$ for all $i$.
\end{proof}

\begin{remark}
The conclusion of Lemma~\ref{basechangefrob2} would not hold true if we used the derived variant of $(-)^{F=1}$ in its formulation. For example, if $n=1$ and $M = \mathbf{Z}/p[F]$, then $(M \widehat{\otimes}_{\mathbf{Z}/p} \mathcal{O}_C^\flat)[1/d] \simeq \widehat{\mathcal{O}_C^\flat[F]}[1/d]$, and one checks that the endomorphism $F-1$ has a non-trivial cokernel on this vector space.
\end{remark}

\begin{corollary}
\label{UnramifiedCrit}
Let $M \in D(\mathbf{Z}/p^n[F])$. Then the  $\mathrm{Gal}(\overline{\mathbf{Q}_p}/\mathbf{Q}_p)$-module 
\[ \left(H^{0}\left(M \widehat{\otimes^L_{\mathbf{Z}/p^n}} W_n(\mathcal{O}_C^\flat)\right)[1/d]\right)^{F=1}\]
is unramified, where $(-)^{F=1}$ is interpreted in the non-derived sense.
\end{corollary}
\begin{proof}
Lemma~\ref{basechangefrob1} translates this to the unramifiedness of $S(H^0(M))$, which was shown in Lemma~\ref{basechangefrob2}.
\end{proof}

Finally, let us record a bound on the inertia invariants of Tate twists over $\mathbf{Q}_p$:

\begin{lemma}
\label{InertiaInvts}
Let $I$ be the inertia group of $\mathbf{Q}_p$. Fix integers $n,i$ and $n \gg 0$ (in fact, $n \geq v_p(i) + 2$ shall suffice). 
\begin{enumerate}
\item If $(p-1) \nmid i$, then $H^0(I, \mathbf{Z}/p^n(i)) = 0$. 
\item Assume $(p-1) \mid i$ and $p$ is odd. Then $\#H^0(I, \mathbf{Z}/p^n(i)) = p^{v_p(i)+1}$.
\item Assume $p=2$. If $i$ is odd then $\#H^0(I, \mathbf{Z}/2^n(i)) = 2$; if $i$ is even, then $\# H^0(I, \mathbf{Z}/2^n(i)) = 2^{v_2(i) + 2}$.
\end{enumerate}
In particular, if $M$ is a $\mathbf{Z}_p$-module regarded as a trivial $I$-representation, then $H^0(I, M(i))$ is annihilated by the number $w(\mathbf{Z}_p,i)$ from Definition~\ref{defJetspace} (which equals $\# H^0(I, \mathbf{Z}/p^n(i)))$  as computed above).
\end{lemma}
\begin{proof}
\begin{enumerate}
\item By devissage, it is enough to show that $H^0(I,\mathbf{F}_p(i)) = 0$ when $(p-1) \nmid i$, or equivalently that $\mathbf{F}_p(i)$ is a non-trivial representation of $I$ for such $i$. But $\mathbf{F}_p(i)$ is the representation determined by $I \twoheadrightarrow \mathbf{Z}_p^* \twoheadrightarrow \mathbf{F}_p^* \xrightarrow{ (-)^i } \mathbf{F}_p^*$. As $\mathbf{F}_p^*$ is a cyclic group of order $p-1$, this representation is nonzero when $(p-1) \nmid i$.

\item We first claim that the image of the representation $I \twoheadrightarrow \mathbf{Z}_p^* \twoheadrightarrow (\mathbf{Z}/p^n)^* \xrightarrow{ (-)^i } (\mathbf{Z}/p^n)^*$ corresponding to $\mathbf{Z}/p^n(i)$ is exactly the subgroup $1 + p^{v+1} \mathbf{Z}/p^n\mathbf{Z}$. Indeed, this follows by consideration of multiplication by $i$ on the short exact sequence 
\[ 1 \to  1 + p \mathbf{Z}/p^n \to (\mathbf{Z}/p^n)^* \to (\mathbf{Z}/p)^* \to 1\]
and using the identification $1 + p\mathbf{Z}/p^n \simeq \mathbf{Z}/p^{n-1}$ coming via the logarithm (as $p$ is odd). It then follows that $H^0(I, \mathbf{Z}/p^n(i))$ identifies with the subgroup 
\[ \{x \in \mathbf{Z}/p^n \mid (1 + p^{v+1} b)x = x \  \forall b \in \mathbf{Z}/p^n\},\]
which is exactly $p^{n-v-1} \mathbf{Z}/p^n\mathbf{Z}$, which has the predicted order.

\item Exercise.
\end{enumerate}

The claim for a general module $M$ follows immediately by applying the above calculations to the cyclic submodule of $M(i)$ generated by an invariant element (and noting that $H^0(I,\mathbf{Z}_p(i)) = 0$).
\end{proof}

\subsection{Proof of the main theorem}
\label{ss:ProofTorsionBoundJet}

\begin{proof}[Proof of Theorem~\ref{torsionboundjet}]
Let us first prove the theorem over the $A_{\inf}$-prism from Notation~\ref{FontainePrism}. We exploit the fact that $J\mathrm{BGL}_{r,/A_{\inf}}$ is defined over $\mathbf{Z}_p$: we have $J\mathrm{BGL}_{r,/A_{\inf}} \simeq JB\mathrm{GL}_{n,/\mathbf{Z}_p} \times_{\mathrm{Spf}(\mathbf{Z}_p)} \mathrm{Spf}(A_{\inf})$ as a (derived) $\delta$-stack. Thus, if we set $M = R\Gamma(JB\mathrm{GL}_{n,/\mathbf{Z}_p}, \mathcal{O}/p^n) \in D(\mathbf{Z}/p^n[F])$, then base change tells us that 
\[ R\Gamma(J\mathrm{BGL}_{r,/A_{\inf}}, \mathcal{O}/p^n)\{i\} \simeq M \widehat{\otimes_{\mathbf{Z}/p^n}^L} W_n(\mathcal{O}_C^\flat)\{i\}\]
in $D(W_n(\mathcal{O}_C^\flat)[F])$. Thus, the image $\overline{c_i}$ of $c_i$ lives in $H^{2i}(M \widehat{\otimes_{\mathbf{Z}/p^n}^L} W_n(\mathcal{O}_C^\flat)\{i\})$. Now this group is $d$-torsionfree by Lemma~\ref{basechangefrob1}, so it is enough to show that the corresponding element 
\[ \overline{c_i}' \in H^{2i}(M \widehat{\otimes_{\mathbf{Z}/p^n}^L} W_n(\mathcal{O}_C^\flat)\{i\}[1/d])\] 
is annihilated by $w(\mathbf{Z}_p,i)$. But we have 
\[ H^{2i}(M \widehat{\otimes_{\mathbf{Z}/p^n}^L} W_n(\mathcal{O}_C^\flat)\{i\}[1/d]) \simeq H^{2i}(M \widehat{\otimes_{\mathbf{Z}/p^n}^L} W_n(\mathcal{O}_C^\flat)[1/d])(i)\]
compatibly with Frobenius and Galois actions by \cite[Example 4.24]{BMS1}.  Moreover, the element $\overline{c_i}'$ is Galois invariant and fixed by the Frobenius. The claim now follows from Corollary~\ref{UnramifiedCrit} and Lemma~\ref{InertiaInvts}.

It remains to pass from the $A_{\inf}$ prism to general bounded prisms $(A,I)$. By naturality of the theorem in the prism $(A,I)$, we immediately know the conclusion for any bounded prism that receives a map from the $A_{\inf}$ prism. Moreover, again by naturality, we may assume that $(A,I)$ is transversal (i.e., $A/I$ is $p$-torsionfree): every bounded prism receives a map from a transversal prism by \cite[Proposition 2.4.1]{BhattLurieAPC}. By taking a coproduct with the $A_{\inf}$-prism (see \cite[Proposition 2.4.5]{BhattLurieAPC}), we can find a $(p,I)$-completely faithfully flat map $(A,I) \to (B,IB)$ of prisms such that the theorem holds true for $(B,IB)$. It  remains to descend the vanishing along $A \to B$. We thus find ourselves in the following abstract situation: we have an object $M := R\Gamma(JB\mathrm{GL}_{n,/\mathbf{Z}_p}, \mathcal{O}/p^n)[2i] \in D(\mathbf{Z}/p^n)$ and an element $m \in H^0(M \widehat{\otimes^L_{\mathbf{Z}/p^n}} A/p^n)\{i\}$ whose image in $H^0(M \widehat{\otimes^L_{\mathbf{Z}/p^n}} B/p^n)\{i\}$ vanishes, and we would like to conclude $m=0$. It thus suffices to show that 
\[ \alpha: H^0(M \widehat{\otimes^L_{\mathbf{Z}/p^n}} A/p^n)\{i\} \to  H^0(M \widehat{\otimes^L_{\mathbf{Z}/p^n}} B/p^n)\{i\} \]
is injective.
As $(A,I)$ is transversal, the map $\mathbf{Z}/p^n \to A/(p^n,I^k)$ is flat for all $k$, whence 
\[ H^i(M \otimes^L_{\mathbf{Z}/p^n} A/(p^n,I^k)) \simeq H^i(M) \otimes_{\mathbf{Z}/p^n} A/(p^n,I^k)\]
 (and similarly with $A$ replaced with $B$).  This implies that the Milnor sequence collapses to an isomorphism
\[   H^0(M \widehat{\otimes^L_{\mathbf{Z}/p^n}} A/p^n) = \lim_k H^0(M) \otimes_{\mathbf{Z}/p^n} A/(p^n,I^k) \]
and similarly with $A$ replaced with $B$. The map $\alpha$ is thus an inverse limit over $k$ of the maps
\[ \alpha_k:H^0(M) \otimes_{\mathbf{Z}/p^n} A/(p^n,I^k)\{i\} \to H^0(M) \otimes_{\mathbf{Z}/p^n} B/(p^n,I^k)\{i\} \]
induced from the map $A \to B$ via base change. But $A \to B$ is $(p,I)$-completely faithfully flat, so the map $A/(p^n,I^k) \to B/(p^n,I^k)$ appearing above is faithfully flat, whence each $\alpha_k$ is injective, and thus the same is true for $\alpha=\lim_k \alpha_k$, as desired.
\end{proof}

\section{Vanishing of Chern classes of prismatic crystals}
\label{sec:VanPrismChern}

Recall that the relative prismatic cohomology constructed in \cite{BMS2,BhattScholzePrisms}  provides a generalization of crystalline cohomology to schemes of mixed characteristic; one of main new features of prismatic cohomology (in comparison to crystalline cohomology) is the \'etale specialization, linking it to \'etale cohomology of the generic fibre in suitable situations. 

In this section, we analyze the (relative) prismatic Chern classes  of perfect complexes on schemes in mixed characteristic that admit a lifting to prismatic crystals, giving a mixed characteristic analog (Theorem~\ref{ChernPrismCrysTors}) of the crystalline vanishing result in Theorem~\ref{CrysChernVanish}. The \'etale specialization prevents these Chern classes from being $0$ on the nose (Example~\ref{Chernexnonzero}), but we shall show they are torison with uniformly bounded torsion exponents.   Our basic setup is the following.

\begin{notation}
Fix a bounded prism $(A,I)$; write $\overline{A} = A/I$. Write $(A,I)_\Prism$ for the slice category of bounded prisms over $(A,I)$. We shall write $\mathcal{O}$ and $\overline{\mathcal{O}}$ for the presheaves on $(A,I)_\Prism$ given by $(B,IB) \mapsto B$ and $(B,IB) \mapsto \overline{B}$ respectively. The category $(A,I)_\Prism$ carries natural topologies (such as Zariski, \'etale, or even flat topologies). However, unless otherwise specified, we shall endow $(A,I)_\Prism$ with the indiscrete topology, so sheaves are just presheaves. For any presheaf $X$ on $p$-nilpotent $A/I$-algebras, write $(X/A)_\Prism$ for the prismatic site of $X$ relative to $A$, i.e., the category of $(B,IB) \in (A,I)_\Prism$ equipped with a map $\mathrm{Spf}(\overline{B}) \to X$ over $\overline{A}$. 
\end{notation}

\subsection{Prismatic Chern classes}

We are interested in the relative prismatic cohomology $H^*_\Prism(X/A)$ of $p$-adic formal schemes $X/\overline{A}$. For our applications, it will be convenient to realize this cohomology theory via a slightly stacky perspective.

\begin{construction}[The prismatization functor]
\label{Cons:Prismatize}
For a presheaf $X$ on $p$-nilpotent $A/I$-algebras, write $(X/A)^\Prism$ for presheaf on $(A,I)_\Prism$ given by 
\[ (X/A)^\Prism(B,IB) = X(\overline{B}) := \lim_n X(\overline{B}/p^n) = \mathrm{Map}_{\overline{A}}(\mathrm{Spf}(\overline{B}),X).\] 
If $X$ is a sheaf for the Zariski/\'etale topology, so is $(X/A)^\Prism$.  The partial slice category $\mathcal{C}_X := \left((A,I)_\Prism\right)_{/ (X/A)^\Prism}$ is identified with the category of bounded prisms $(B,IB) \in (A,I)_\Prism$ equipped with an $\overline{A}$-map $\mathrm{Spf}(\overline{B}) \to X$; to this category, we can associate $\infty$-categories
\[ \mathrm{Perf}( (X/A)^\Prism, \mathcal{O}) := \lim_{\mathcal{C}_X} \mathrm{Perf}(B) \quad \text{and} \quad \mathrm{Perf}( (X/A)^\Prism, \overline{\mathcal{O}}) := \lim_{\mathcal{C}_X} \mathrm{Perf}(\overline{B}) \]
of {\em prismatic crystals} and {\em Hodge--Tate crystals} on $X/A$ respectively. By the explicit description of $\mathcal{C}_X$, these agree with the $\infty$-categories of crystals of perfect complexes on the ringed sides $( (X/A)_\Prism, \mathcal{O}_\Prism)$ and $( (X/A)_\Prism, \overline{\mathcal{O}_\Prism})$ respectively, so the terminology is justified. In particular, we have natural identifications 
\[ R\Gamma_\Prism(X/A) \simeq R\Gamma( (X/A)^\Prism,\mathcal{O}) \quad \text{and} \quad R\Gamma_{\overline{\Prism}}(X/A) \simeq R\Gamma( (X/A)^\Prism, \overline{\mathcal{O}}) \]
of $E_\infty$-algebras over $A$ and $\overline{A}$ respectively.
\end{construction}

\begin{remark}
\label{CompatPrismatization}
Let us explain the terminology ``prismatization'' used for Construction~\ref{Cons:Prismatize}. Given a presheaf $X$ on $p$-nilpotent $A/I$-algebras, we introduced in \cite[\S 5]{BhattLuriePrismatization} the relative prismatization $(X/A)^\Prism$ as a functor on $(p,I)$-nilpotent $A$-algebras, and endowed it with a $\delta$-structure, compatible with that of $A$. Any such functor $F$ induces a functor $(A,I)_\Prism$: send $(B,IB) \in (A,I)_\Prism$ to $\mathrm{Map}_{A,\delta}(\mathrm{Spf}(B),F)$. One can check that the functor $(X/A)^\Prism$ from \cite{BhattLuriePrismatization} is carried to the functor $(X/A)^\Prism$ from Construction~\ref{Cons:Prismatize} under this procedure, justifying the notational overload.
\end{remark}

\begin{remark}[From perfect complexes on $X$ to Hodge--Tate crystals]
\label{PerftoHT}
For any $p$-adic formal scheme $X$, we can naturally associate Hodge--Tate crystals to perfect complexes on $X$. More precisely, there is a natural symmetric monoidal functor
\[ \mathrm{Perf}(X) \to \mathrm{Perf}( (X/A)^\Prism, \overline{\mathcal{O}})\]
given by taking a limit of the pullback functors $\mathrm{Perf}(X) \to \mathrm{Perf}(\overline{B})$ over the category $\mathcal{C}_X$. As we shall see in Construction~\ref{PrismChernClassDef}, the Chern class map factors over this map, i.e., any Hodge--Tate crystal on $X$ admits Chern classes.
\end{remark}

To construct Chern classes, we shall use the following basic calculation, deforming Theorem~\ref{DeligneCalc} to a prismatic context:

\begin{theorem}[Prismatic cohomology of $\mathrm{BGL}_n$]
\label{PrismCohBGL}
Fix an integer $n \geq 1$. Let $X = \mathrm{Vect}_n(-/\overline{A})$ be the classifying stack of rank $n$ vector bundles on $p$-nilpotent $\overline{A}$-algebras. Then $H^{\mathrm{odd}}_\Prism(X/A)$ vanishes, and there is a natural identification of graded rings
\[ H^{2*}_\Prism(X/A)\{*\} = A[c_1,...,c_n],\]
where $c_i \in H^{2i}_\Prism(X/A)\{i\}$ has degree $2i$.
\end{theorem}
\begin{proof}
This is proven for syntomic cohomology in \cite[\S 9.3]{BhattLurieAPC}; the arguments there also work (more easily) for relative prismatic cohomology. Alternately, a direct argument over the Breuil--Kisin or $A_{\inf}$-prisms was given in \cite[Corollary 6.1.13]{KPTotaro}.
\end{proof}

We can now attach Chern classes to Hodge--Tate crystals:

\begin{construction}[Prismatic Chern classes for Hodge--Tate crystals]
\label{PrismChernClassDef}
Fix a presheaf $X$ on $p$-nilpotent $A/I$-algebras. For any $E \in \mathrm{Perf}( (X/A)^\Prism, \overline{\mathcal{O}})$, we shall construct Chern classes $c_i(E) \in H^{2i}_\Prism(X)\{2i\}$ for $i \geq 1$. Via the map in Remark~\ref{PerftoHT}, this also yields Chern classes for perfect complexes on $X$ itself\footnote{\label{ChernHTfactor}Thus, we are implicitly asserting that the Chern class construction $\mathrm{Perf}(X) \xrightarrow{c_i} H^{2i}_\Prism(X/A)\{i\}$ factors over the pullback $\mathrm{Perf}(X) \to \mathrm{Perf}( (X/A)^\Prism, \overline{\mathcal{O}})$. In the stack-theoretic language of \cite{BhattLuriePrismatization}, the latter map identifies with the pullback $\mathrm{Perf}(X) \to \mathrm{Perf}( (X/A)^{HT})$ along $(X/A)^{HT} \to X$; using this language, we can explain why this a priori surprising (to us, at least) factorization of the Chern class is in fact reasonable. First, if $X$ is smooth over $\overline{A}$, then one can show that pullback along $(X/A)^{HT} \to X$ induces an equivalence on continuous $K$-theory, i.e., on the $K$-theory of the mod $p^n$-reductions. Perhaps more conceptually, given any map $X \to Y$ of presheaves over $\overline{A}$, its prismatization $(X/A)^\Prism \to (Y/A)^\Prism$ only depends on the composition $(X/A)^{HT} \to X \to Y$; this follows by unwinding definitions. Applying this observation with $Y=\mathrm{Vect}(-/\overline{A})$ shows that for any vector bundle $E$ on $X$, the map $[E]^\Prism:(X/A)^\Prism \to (\mathrm{Vect}(-/\overline{A})/A)^\Prism$ induced by $E$, which determines the Chern classes of $E$ by universality, only depends on the pullback of $E$ to $(X/A)^{HT}$.}. 

We shall explain the construction for vector bundles; the extension to all perfect complexes is obtained via $K$-theory as in Construction~\ref{conscryschernperf}.  Fix $E \in \mathrm{Vect}_n( (X/A)^\Prism, \overline{\mathcal{O}})$. Then we have an associated classifying map $[E]^\Prism:(X/A)^\Prism \to \mathrm{Vect}_n(-/\overline{A})^\Prism$ determined as follows: for any $(B,IB) \in (A,I)_\Prism$,  a point of $(X/A)^\Prism(B,IB)$ is determined by a map $\eta:\mathrm{Spf}(\overline{B}) \to X$, which in turn determines a vector bundle $E(\eta) \in \mathrm{Vect}_n(\overline{B}) = \mathrm{Vect}_n(-/\overline{A})^\Prism(B,IB)$ by definition of $E$. Pulling back the universal classes from Theorem~\ref{PrismCohBGL} then gives the desired classes $c_i(E) := ([E]^\Prism)^*(c_i) \in H^{2i}_\Prism(X/A)\{2i\}$.
\end{construction}

\subsection{Vanishing for prismatic crystals}

Our goal is to prove the following result, showing that Chern classes of perfect complexes that admit prismatic crystal structures are torsion with uniformly bounded exponents, at least\footnote{We expect it is not necessary to work modulo powers of $p$ in the conclusion of Theorem~\ref{ChernPrismCrysTors}.} modulo any power of $p$. To formulate the statement, write $R\Gamma_\Prism(X/A; \mathbf{Z}/p^n) := R\Gamma_\Prism(X/A) \otimes_{\mathbf{Z}}^L \mathbf{Z}/p^n$. Then our result states:

\begin{theorem}[Chern classes of prismatic crystals are torsion]
\label{ChernPrismCrysTors}
Fix a presheaf $X$ on $p$-nilpotent $A/I$-algebras as well as a Hodge--Tate crystal $E \in \mathrm{Perf}( (X/A)^\Prism, \overline{\mathcal{O}})$. If $E$ lifts to a prismatic crystal (i.e., lifts to $\mathrm{Perf}( (X/A)^\Prism, \mathcal{O})$), then  
\[ w(\mathbf{Z}_p,i) c_i(E) = 0 \quad \text{in} \quad  H^{2i}_\Prism(X/A; \mathbf{Z}/p^n)\{i\}\]
for all $n \geq 0$ and $i > 0$. 

If we assume that $H^{2i-1}_\Prism(X/A)$ has bounded $p^\infty$-torsion, then $c_i(E) \in H^{2i}_\Prism(X/A)\{i\}$ is annihilated by $w(\mathbf{Z}_p,i)$ for each $i > 0$. 
\end{theorem}

\begin{remark}[When does prismatic cohomology have bounded $p^\infty$-torsion?]
\label{BoundedTors}
The bounded $p^\infty$-torsion condition in the last part of Theorem~\ref{ChernPrismCrysTors} is satisfied in a number of cases of interest. For instance, it holds true when $X/\overline{A}$ is a smooth proper $p$-adic formal scheme and $(A,I)$ satisfies one of the following conditions:
\begin{enumerate}
\item $A$ is noetherian: 
\item $(A,I) = (A_{\inf}(\mathcal{O}_C),\ker(\theta))$ for a complete and algebraically closed extension $C/\mathbf{Q}_p$.
\end{enumerate}
Indeed, (1) follows as $R\Gamma_\Prism(X/A)$ is $A$-perfect and $A$ is noetherian; (2) is treated in \cite[Theorems~1.8 and 4.13]{BMS1}). Another important example is provided by taking $(A,I)$ to be a Breuil--Kisin prism, and $X/\overline{A}$ to be a Hodge-proper stack (see \cite{KPTotaro}); base changing these examples along faithfully flat maps of prisms provides more examples. 
\end{remark}

Before preceeding, let us note that Theorem~\ref{ChernPrismCrysTors} is weaker than the corresponding statement in the crystalline case: the latter gives vanishing on the nose, while the former only gives vanishing up to (controlled) torsion. This is not merely an artifact of the proof but in fact forced by nature:

\begin{remark}[The Chern class of a prismatic crystal can be nonzero]
\label{Chernexnonzero}
In the setup of Theorem~\ref{ChernPrismCrysTors}, the Chern classes $c_i(E)$ can be nonzero, i.e., the factors $w(\mathbf{Z}_p,i)$ are unavoidable. For example, say $p=2$ and we work in rank $1$. Theorem~\ref{ChernPrismCrysTors} shows that $2c_1(L) = 0$ when $L$ is a line bundle $X$ that lifts to a prismatic crystal. We shall find an example of such an $L$ where $c_1(L) \neq 0$. 

Say $(A,I) = (A_{\inf}(\mathcal{O}_C),\ker(\theta))$ for a complete and algebraically closed extension $C/\mathbf{Q}_p$ and $X/\mathcal{O}_C$ is an Enriques surface with $\pi_1(X) = \mathbf{Z}/2$. If $f:\tilde{X} \to X$ denotes the universal cover, then $L := f_* \mathcal{O}_{\tilde{X}}/\mathcal{O}_X$ is a line bundle on $X$. Moreover, since $f$ is finite \'etale, the corresponding Hodge--Tate crystal $L_{\overline{\mathcal{O}}} \in \mathrm{Perf}( (X/A)^\Prism, \overline{\mathcal{O}})$  lifts to a prismatic crystal: indeed, a lift is provided by $g_* \mathcal{O}/\mathcal{O}$, where $g:(\tilde{X}/A)^\Prism \to (X/A)^\Prism$ is the finite \'etale map on relative prismatizations induced by $f$. However, we claim that $c_1(L) \in H^2_\Prism(X/A)\{1\}$ is nonzero. In fact, by the compatibility of the prismatic and \'etale Chern classes under the \'etale comparison theorem \cite[Remark 9.2.6]{BhattLurieAPC}, it suffices to show that $L_C \in \mathrm{Pic}(X_C)$ has a non-trivial first Chern class $c_1(L_C)$ in $H^2(X_C, \mathbf{Z}_2(1))$, which follows by the argument in Example~\ref{Torsc1Ex}.
\end{remark}

We now turn to the proof of Theorem~\ref{ChernPrismCrysTors}. It will rely on the following, which is roughly a prismatic analog of Construction~\ref{SharpDef}:

\begin{construction}[The $\sharp$-functor from $A$-stacks to presheaves on $(A,I)_\Prism$]
For a presheaf $X$ on $(p,I)$-nilpotent $A$-algebras, write $X^\sharp$ for the presheaf on $(A,I)_\Prism$ given by $X^\sharp(B,IB) = X(B) := \mathrm{Map}_A(\mathrm{Spf}(B),X)$. If we write $\overline{X} = X \times_{\mathrm{Spf}(A)} \mathrm{Spf}(\overline{A})$ for the induced functor $p$-nilpotent $A$-algebras, then we have a natural
\[ \eta_X:X^\sharp \to (\overline{X}/A)^\Prism\]
induced via the obvious map $B \to \overline{B}$ for any $(B,IB) \in (A,I)_\Prism$.
\end{construction}

We shall analyse this construction using the arithmetic jet space from Definition~\ref{defJetspace}.

\begin{example}[$\mathrm{Spf}(R)^\sharp$ for smooth $R$]
\label{sharpcalcsmooth}
Let $R$ be a $(p,I)$-completely smooth $A$-algebra, and let $X=\mathrm{Spf}(R)$, regarded as a presheaf on $(p,I)$-nilpotent $A$-algebras. Then the presheaf $X^\sharp$ on $(A,I)_\Prism$ has an explicit description: it is representable, with representing object $(F,IF)$, where $F$ is the free $(p,I)$-complete $\delta$-$A$-algebra on $R$. In fact, this is immediate from the definition once we know that $(F,IF)$ is a bounded prism, i.e., $F$ is $I$-torsionfree and $F/I$ has bounded $p^\infty$-torsion. Both these properties would follow from the stronger assertion that $F$ is $(p,I)$-completely flat over $A$.  This assertion can be checked by devissage. Indeed,  it is clear when $R$ equals $R_0 := A[x_1,...,x_n]^{\wedge}_{(p,I)}$ via the explicit description of free $\delta$-rings, as in \cite[Lemma 2.11]{BhattScholzePrisms}. As $\delta$-structures lift uniquely along \'etale maps \cite[Lemma 2.18]{BhattScholzePrisms}, the assignment $S \mapsto F(S)$ forms a quasi-coherent sheaf for the \'etale topology on $\mathrm{Spf}(R_0)$ or $\mathrm{Spf}(R)$, so the claim follows in general as $\mathrm{Spf}(R)$ admits an open cover consisting of \'etale formal schemes over $\mathrm{Spf}(R_0)$. 
\end{example}

We need the following calculational property of this construction, analogous to Proposition~\ref{cohcohstack} and extending Example~\ref{sharpcalcsmooth}.

\begin{proposition}
\label{prismatization2}
Let $X$ be a presheaf on $(p,I)$-nilpotent $A$-algebras that admits an \'etale hypercover by $(p,I)$-completely smooth formal $A$-schemes. Then there is a natural isomorphism 
\[ R\Gamma(X^\sharp,\mathcal{O}) \simeq R\Gamma(JX, \mathcal{O})\]
of $E_\infty$-$A$-algebras.
\end{proposition}
\begin{proof}
When $X$ is representable by a $(p,I)$-completely smooth formal affine $A$-scheme, the claim follows from Example~\ref{sharpcalcsmooth}. To deduce the general case, we argue via descent. First, observe that the functor $X \mapsto X^\sharp$ from $(p,I)$-completely smooth formal affine $A$-schemes to presheaves on $(p,I)$-nilpotent $A$-algebras clearly preserves finite limits. Moreover, it also preserves \'etale maps/covers by the existence and uniqueness of lifts of $\delta$-structures along $(p,I)$-completely \'etale maps. Thus, $F(-) := R\Gamma( (-)^\sharp,\mathcal{O})$ is an \'etale sheaf. Similarly, as the formation of $JX$ is compatible with \'etale localization on $X$, the functor $G(-) := R\Gamma(J(-), \mathcal{O})$ is also an \'etale sheaf. As we have already constructed a natural isomorphism $F(X) \simeq G(X)$ for $(p,I)$-completely smooth formal affine $A$-schemes, the claim follows for general $X$ as in the statement of the proposition.
\end{proof}


Let us put everything together:

\begin{proof}[Proof of Theorem~\ref{ChernPrismCrysTors}]

We first explain the proof when $E$ is a Hodge--Tate crystal of vector bundles. Without loss of generality, we may assume $E$ has a fixed rank, say $r$. Write $Y=\mathrm{Vect}_r(-/A)$ and $\overline{Y} = Y_{\overline{A}} = \mathrm{Vect}_r(-/\overline{A})$ for the displayed stacks over $\mathrm{Spf}(A)$ and $\mathrm{Spf}(\overline{A})$ respectively.  The existence of $E$ gives a natural map 
\[ (X/A)^\Prism \to (\overline{Y}/A)^\Prism \]
of presheaves on $(A,I)_\Prism$ by the following procedure: given $(B,IB) \in (A,I)_\Prism$ and a map $\mathrm{Spf}(\overline{B}) \to X$ determining a point of $(X/A)^\Prism(B,IB)$, we obtain a vector bundle $E(\overline{B}) \in \mathrm{Vect}_r(B)$ by evaluating the Hodge--Tate crystal $E$, and thus a point of $(\overline{Y}/A)^\Prism(B,IB)$. Specifying a lift of $E$ to a prismatic crystal in vector bundles is tantamount to factoring the previous map as 
\[ (X/A)^\Prism \to Y^\sharp \xrightarrow{\eta_Y}  (\overline{Y}/A)^\Prism. \]
It is therefore enough to show that the image of the universal Chern class $c_i \in H^{2i}_\Prism(\overline{Y}/A; \mathbf{Z}/p^n)\{i\}$ under $\eta_Y^*$ is killed by $w(\mathbf{Z}_p,i)$ for $n \geq 0$. Using Proposition~\ref{prismatization2} to realize $R\Gamma(Y^\sharp,\mathcal{O}/p^n) \simeq R\Gamma(JY, \mathcal{O}/p^n)$, the claim follows from Theorem~\ref{torsionboundjet}. The general case is deduced using Thomason's theorem, as in Theorem~\ref{CrysChernVanish}.
\end{proof}

\begin{remark}[Other results on torsion bounds for Chern classes]
\label{RelatedTorsion}
Fix a complete and algebraically closed extension $C/\mathbf{Q}_p$, a smooth proper variety $X_C$ admitting a smooth proper integral model $X/\mathcal{O}_C$, and a vector bundle $E \in \mathrm{Vect}(X_C)$. We then have Chern classes $c_i(E) \in H^{2i}(X_C, \mathbf{Z}_p(i))$ in \'etale cohomology.  Theorem~\ref{ChernPrismCrysTors} (and Remark~\ref{BoundedTors}) shows that
\[ w(\mathbf{Z}_p,i) \cdot c_i(E) = 0 \in H^{2i}(X_C, \mathbf{Z}_p(i)) \]
provided $E$ admits a certain additional sturcture (namely, an extension to a perfect complex on $X$ that lifts to a prismatic crystal). Under different constraints on $E$, this vanishing result has some antecedents in the literature. Two instances are:
\begin{enumerate}
\item When $E$ is a Gauss--Manin connection for a semistable fibration over $X_C$ and $c_i^{\mathrm{CH}}(E) \in \mathrm{CH}^i(X_C)$ denotes the $i$-th Chern class in the Chow group of codimension $i$ cycles, Maillot--R\"{o}ssler \cite{MaillotRosslerChern} conjecture  $w(\mathbf{Z}_p,i) \cdot c^{\mathrm{CH}}_i(E) = 0$ in $\mathrm{CH}^i(X_C) \otimes_{\mathbf{Z}} \mathbf{Z}_{(p)}$;  they also prove the corresponding assertion in cohomology (attributing the case of smooth fibrations to Grothendieck).  Prior to this work, in light of Reznikov's solution \cite{ReznikovBloch} to Bloch's conjecture, Esnault had already formulated (see \cite{EsnaultFlatCC} and \cite{EVChernWeight1}) the weaker question that $c_i(E) = 0$ in $\mathrm{CH}^i(X) \otimes_{\mathbf{Z}} \mathbf{Q}$, and in fact integers closely related to $w(\mathbf{Z}_p,i)$ also appear in the proof in \cite[\S 2]{EVChernWeight1} (where they arise  via a Chern character calculation).

\item Say $G$ is a finite group and $X=BG_{\mathcal{O}_C}$. Then $E \in \mathrm{Vect}(X_C)$ corresponds to a $C$-linear representation of $G$; note that any such representation has an $\mathcal{O}_C$-stable lattice, so $E$ admits a lift to $\mathrm{Vect}(X)$. Grothendieck proved \cite[Theorem 4.8]{GrothendieckChernDiscrete} that $n(E,i) \cdot c_i(E) = 0$ in $H^{2i}(X_C, \mathbf{Z}_p(i))$ where $n(E,i)$ is an explicit integer depending on the ramification needed to define $E$ (as a representation of $G$ on a $C$-vector space) and the integer $i$. In the special case where $E$ can be defined over a base field unramified over $\mathbf{Q}_p$, the bundle $E$ does indeed lift to a prismatic crystal on $X$ and we also have $n(E,i) = w(\mathbf{Z}_p,i)$ (up to factors coprime to $p$), so Theorem~\ref{ChernPrismCrysTors} is compatible with Grothendieck's theorem in this case. 

For the same $X$ as above, the following example illustrates the use of Theorem~\ref{ChernPrismCrysTors} in concluding that not all vector bundles on $X_C$ lift to prismatic crystals (this particular example also admits an elementary proof, but it illustrates the obstruction):

\begin{example}
\label{NoPrismaticLift}
Say $p=3$, $G=\mathbf{Z}/3$ and $X=BG_{\mathcal{O}_C}$ as above. Since $\underline{\mathrm{Hom}}(G,\mathbf{G}_m) = \mu_3$ as group schemes over any base, we have $\mathrm{Pic}(X) \simeq \mathrm{Pic}(X_C) \simeq \mu_3(C)$. Pick $L \in \mathrm{Pic}(X)$ giving a nonzero element of this group.  Note that $L_C$ admits the structure of an infinitesimal crystal on $X_C$: in fact, every bundle on $X_C$ admits such a structure as $X_C$ is \'etale over $C$. Moreover, the argument in Example~\ref{Torsc1Ex} shows that $c_1(L_C) \in H^2(X_C, \mathbf{Z}_3(1))$ is nonzero, whence $c_1(L) \in H^2_\Prism(X/A)\{1\}$ is also nonzero. Since $w(\mathbf{Z}_3,1) = 1$, it follows from Theorem~\ref{ChernPrismCrysTors} that $L$ cannot lift to a ($\varphi$-twisted) prismatic crystal on $X$. Using a suitable Godeaux--Serre approximation, one can also find a similar example where $X/\mathcal{O}_C$ is a smooth projective variety instead of a stack.
\end{example}
\end{enumerate}
\end{remark}

\begin{remark}[A question on vanishing in syntomic cohomology]
\label{SynMixedQuestion}
Based on the Maillot--Rossler(--Esnault) conjecture in Remark~\ref{RelatedTorsion} (1) and the fact that syntomic cohomology can be regarded as $p$-adic \'etale motivic cohomology, the analogy of prismatic $F$-gauges with $p$-adic motives suggests the following: 
\begin{question}
 Say $X/\mathbf{Z}_p$ is a proper regular $p$-adic formal scheme and $E \in \mathrm{Perf}(X)$ is a perfect complex on $X$ admitting the structure of a prismatic $F$-gauge. Is $w(\mathbf{Z}_p,i) \cdot c_i(E) = 0$ in $H^{2i}_{\mathrm{Syn}}(X, \mathbf{Z}_p(i))$? An example of such an $E$ is provided by a vector bundle on $X$ that admits the structure of a prismatic $F$-crystal.
\end{question}
For $X$ having characteristic $p$, such a vanishing indeed holds true: in fact, the syntomic Chern classes themselves vanish without the $w(\mathbf{Z}_p,i)$ factor, see Remark~\ref{SynVanCharpRemark} (2). On the other hand, a variant of Example~\ref{Chernexnonzero} shows that the factors $w(\mathbf{Z}_p,i)$ are necessary in general.

As in Remark~\ref{SynVanCharpRemark}, one might more ambitiously also ask the above question without the $F$-structure, i.e., for all (absolute) prismatic crystals.
\end{remark}

\subsection{Vanishing for $\varphi$-twisted prismatic crystals}
\label{TwistedPrismaticCrystals}
We have motivated  Theorem~\ref{ChernPrismCrysTors} as a prismatic counterpart of Theorem~\ref{CrysChernVanish}, but the formulations are not obviously related: besides the factors of $w(\mathbf{Z}_p,i)$ appearing in the former (which are unavoidable by Example~\ref{Chernexnonzero}), the former concerns Hodge--Tate crystals on $X$ admitting additional structure, while the latter concerns perfect complexes on $X$ admitting additional structure. Nonetheless, one can extract a result from Theorem~\ref{ChernPrismCrysTors} that looks formally quite similar to Theorem~\ref{CrysChernVanish}; we explain how do so in the rest of this remark, using the language of derived prismatization. (Thus, in this subsection, $(X/A)^\Prism$ refers to a presheaf on all $(p,I)$-nilpotent animated $A$-algebras, as constructed in \cite{BhattLuriePrismatization}; this is compatible with the rest of this section, see Remark~\ref{CompatPrismatization}.)

\begin{construction}[Factoring the Frobenius on the Hodge--Tate stack]
Let $X$ be a presheaf on $(p,I)$-nilpotent animated $A$-algebras. We shall construct a map allowing us to extract perfect complexes on $X$ from prismatic data. More precisely, we shall construct the following commutative diagram of presheaves on $(p,I)$-nilpotent animated $A$-algebras:
\begin{equation}
\label{PrismCrysTwist}
 \xymatrix{ (X/A)^{HT} \ar[r]^{\pi} \ar[d]^{\tau} & X \ar[d]^\psi \\
(X/A)^\Prism \ar[r]^{\varphi_{X/A}}  & \varphi_A^* (X/A)^\Prism. }
\end{equation}
Here $\tau$ is the tautological closed immersion, $\pi$ is the Hodge--Tate structure map,  and  $\varphi_{X/A}$ is the relative Frobenius for the $\delta$-$A$-prehseaf $(X/A)^\Prism$. To define $\psi$ and the witness to the commutativity of the resulting diagram, we use the following observation: for any $p$-nilpotent animated $A/I$-algebra $R$, if we regard $W(R)$ as an animated $\delta$-$A$-algebra in the standard way, then the animated Cartier--Witt divisor $(I \to A) \otimes_A \varphi_* W(R)$ identifies with $\left(\varphi_* W(R) \xrightarrow{p} \varphi_* W(R)\right)$, whence the mod $I$ reduction of the Frobenius $W(R) \xrightarrow{F} \varphi_* W(R)$ factors as $W(R)/I \to R \xrightarrow{\epsilon_R} \varphi_*  W(R)/p$, where the first map is the canonical projection, while the second map is induced by the Witt vector Frobenius (and already featured in Remark~\ref{stackycrysvan}). Using the functor of points, this factorization constructs the desired maps as well a witness to the commutativity of the diagram. 
\end{construction}

\begin{definition}[$\varphi$-twisted prismatic crystals]
Given a presheaf $X$ on $(p,I)$-nilpotent $A$-algebras, the $\infty$-category of {\em $\varphi$-twisted prismatic crystals} is defined to be $\mathrm{Perf}(\varphi_A^* (X/A)^\Prism)$.  Any  $\varphi$-twisted prismatic crystal has a {\em value on $X$}, which is the perfect complex on $X$ via pullback along $\psi$ in diagram \eqref{PrismCrysTwist}. 
\end{definition}

Let us explain the relation to crystals on the classical crystalline site.

\begin{example}[The crystalline case]
Say $(A,I)$ is a crystalline prism, i.e., $I = (p)$. For any presheaf $X$ on $p$-nilpotent $A$-algebras, we can identify $\varphi$-twisted prismatic crystals with crystals of perfect complexes on $(X/A)_\crys$. In fact, the stack $\varphi_A^* (X/A)^\Prism$ identifies with the crystallization $(X/A)^\crys$ appearing in the stacky approach to crystalline cohomology \cite{DrinfeldStacky}: via transmutation, this follows from the isomorphism of rings stacks 
\[ \left((p) \to A\right) \otimes_A \varphi_* W(\mathbf{G}_a) = \left(\varphi_* W(\mathbf{G}_a) \xrightarrow{p} \varphi_* W(\mathbf{G}_a)\right) =: \mathbf{G}_a^{dR}\]
on $p$-nilpotent $A$-algebras (see \cite[Remark 7.9]{BhattLuriePrismatization} for more). Under this equivalence, the $\psi$-pullback operation on $\varphi$-twisted prismatic crystals identifies with the evaluation at $X$ operation on crystals on $(X/A)_\crys$. 
\end{example}

Thus, over a general bounded prism, a $\varphi$-twisted prismatic crystal on $(X/A)_\Prism$ is a generalization of  the classical notion of a crystal of perfect complexes in the crystalline theory.  We then have:

\begin{corollary}
\label{TwistedPrismaticCrystalsCor}
Fix a presheaf $X$ on $p$-nilpotent $A/I$-algebras as a perfect complex $E \in \mathrm{Perf}(X)$. If $E$ lifts to a $\varphi$-twisted prismatic crystal 
\[ w(\mathbf{Z}_p,i) c_i(E) = 0 \quad \text{in} \quad  H^{2i}_\Prism(X/A; \mathbf{Z}/p^n)\{i\}\]
for all $n \geq 0$ and $i > 0$. 

If we assume that each $H^i_\Prism(X/A)$ has bounded $p^\infty$-torsion, then $c_i(E) \in H^{2i}_\Prism(X/A)\{i\}$ is annihilated by $w(\mathbf{Z}_p,i)$ for each $i > 0$. 
\end{corollary}
\begin{proof}
This follows immediately Theorem~\ref{ChernPrismCrysTors} and the commutative square \eqref{PrismCrysTwist}: if $E$ lifts to a $\varphi$-twisted prismatic crystal (i.e., along $\psi^*$), then the diagram shows that the corresponding Hodge--Tate crystal $\pi^* E$ lifts to a prismatic crystal (i.e., along $\tau^*$), whence $c_i(E) := c_i(\pi^* E)$  (by Footnote~\ref{ChernHTfactor}) satisfies the desired conclusion. 
\end{proof}

In other words, up to the unavoidable factor $w(\mathbf{Z}_p,i)$, we obtain a prismatic counterpart to Theorem~\ref{CrysChernVanish}.

\newpage
\bibliographystyle{amsalpha}
\bibliography{mybib}

\end{document}